\documentclass[11pt]{article}

\usepackage[margin=1in]{geometry}
\usepackage{amsmath,amssymb,amsthm,dsfont}
\usepackage{enumitem}
\usepackage{xcolor}
\usepackage{graphicx}
\usepackage[normalem]{ulem}
\usepackage{comment}
\usepackage{url}
\definecolor{revisiongreen}{HTML}{2E7D32}

\newtheorem{theorem}{Theorem}[section]
\newtheorem{proposition}[theorem]{Proposition}
\newtheorem{lemma}[theorem]{Lemma}
\newtheorem{corollary}[theorem]{Corollary}

\theoremstyle{definition}
\newtheorem{definition}[theorem]{Definition}

\newtheorem{remark}[theorem]{Remark}
\newtheorem{ainotice}[theorem]{AI Notice}

\newcommand{\Cay}{\operatorname{Cay}}
\newcommand{\ev}{\operatorname{ev}}

\newcommand{\id}{\operatorname{id}}

\newcommand{\mathsout}[1]{\text{\sout{\(\displaystyle #1\)}}}

\title{Ordered frequency in relatively hyperbolic groups}
\author{Dominic Bair}
\date{19 October 2026}

\begin{document}

\maketitle

\begin{abstract}
	We study word frequencies along infinite strings whose finite prefixes lie in a regular language of geodesic representatives of group elements.
	Let \((G,\mathcal H)\) be a non-elementary relatively hyperbolic pair, and suppose that, over a finite symmetric generating set, there is a prefix-closed regular language containing exactly one shortest word representing each element of \(G\).
	Modeling this language using a finite-state Markov chain, we define a probability measure on the space of infinite strings whose finite prefixes are words belonging to this language.
	Each such string labels a geodesic ray emanating from the identity.
    We show there is a continuous surjection from this space onto the Bowditch boundary, and the pushforward of the probability measure under this map is a Patterson--Sullivan probability obtained as a weak limit of weighted counts of group elements.
	For a word \(w\), the ordered frequency of \(w\) along a given labeled ray is the limiting proportion of consecutive nonoverlapping tuples of length \(|w|\) that exactly spell \(w\).
	For almost every infinite string under the constructed measure, these frequencies exist simultaneously for all nonempty geodesic words, with no restriction on their lengths.
	Consequently, almost every boundary point with respect to this Patterson--Sullivan probability is the endpoint of at least one such ray along which all these frequencies exist.
	In particular, after possibly enlarging the generating set, these results apply when the peripheral subgroups are virtually abelian.
\end{abstract}

\section{Introduction}

How often does a given word appear in a typical infinite string read by a finite-state automaton?
For a prefix-closed regular language, the infinite strings we have in mind are those whose finite prefixes all belong to the language.
We study this question for regular languages that choose one shortest word representing each element of a relatively hyperbolic group.
For a word of length \(k\), we read the infinite string in successive \(k\)-tuples, starting with its first letter, and ask what proportion of these tuples spell the word.
For example, for a three-letter word we inspect letters 1--3, then 4--6, then 7--9, rather than every possible starting position.
The limiting proportion, when it exists, is the \emph{ordered frequency} of the word.
To say what \emph{typical} means, we construct a Markov probability measure on infinite paths in the automaton, which we will henceforth refer to as a \emph{geodesic combing}.
We show that ordered frequencies exist almost surely for this measure, and identify the distribution of the corresponding endpoints on the boundary of the group.

Our motivation comes from first-passage percolation (FPP), where edges are assigned independent, identically distributed passage times and the time needed to travel between two vertices is the infimum of the total passage times of paths joining them.
On \(\mathbb Z^d\), the existence of a directional time constant under suitable assumptions is one of the early results of the subject; see \cite[Section~1]{BasuMj2022} for an overview.
Basu and Mj prove a corresponding result for word-hyperbolic groups: expected passage times, divided by word distance along a geodesic ray, converge in almost every boundary direction with respect to Patterson--Sullivan measure \cite[Theorem~5.1]{BasuMj2022}.
Their proof uses ordered frequencies to average contributions from finite portions of a geodesic ray.
As they explain in \cite[Section~3.3]{BasuMj2022}, this notion modifies a counting function introduced by Rhemtulla \cite{Rhemtulla1968} and rediscovered by Brooks \cite{Brooks1981}.

Mahan Mj asked whether the time constant result could be extended beyond hyperbolic groups, and suggested considering relatively hyperbolic groups.
The case of virtually abelian peripheral subgroups is particularly appealing: it allows Euclidean geometry inside the peripheral subgroups while retaining hyperbolic geometry relative to them.
The lattice and hyperbolic results suggest asking whether a similar law of large numbers holds in this setting.
Ilya Gekhtman suggested the class with \emph{geodesically completable} peripheral subgroups studied in \cite{GekhtmanTaylorTiozzo2020}, which includes the virtually abelian.
Here we extend the ordered-frequency result to this setting.


\subsection*{Main results}

	We consider a non-elementary relatively hyperbolic pair \((G,\mathcal H)\), meaning that its Bowditch boundary is infinite.
	Fix a finite symmetric generating set \(S\) and a finite geodesic combing of \(\Cay(G,S)\), as in Definition~\ref{def:geodesic-combing}.
	Here the combing is a finite generator-labeled directed graph whose finite paths from its initial vertex give exactly one shortest representative for each group element.

	\par

	\noindent\textbf{A Markov description of the boundary measure.}
	Following Calegari--Fujiwara \cite{CalegariFujiwara2010}, we construct a based Markov probability measure \(\mu^0\) on infinite paths beginning at the combing's initial vertex (Lemma~\ref{lem:based-markov-measure}).
	Each such path labels a unique geodesic ray in \(\Cay(G,S)\) from the identity. We show there is a continuous surjective map sending these labeled paths to the Bowditch compactification. (Lemma~\ref{lem:compactified-bowditch-evaluation}).
	Theorem~\ref{thm:critical-counting-convergence} shows that the normalized weighted-counting measures in Definition~\ref{def:critical-counting-ps-measure} converge weakly to a Patterson--Sullivan probability \(\nu\).
	The pushforward of \(\mu^0\) under this endpoint map is exactly \(\nu\).
	\par

	\noindent\textbf{Ordered frequencies exist almost surely.}
	Theorem~\ref{thm:arbitrary-length-ordered-frequencies} gives a single set of full \(\mu^0\)-measure on which the ordered frequencies exist simultaneously for every nonempty geodesic word, of any finite length.
	For \(\nu\)-almost every boundary point, all these ordered frequencies exist along at least one combing ray ending there (Corollary~\ref{cor:arbitrary-length-boundary-frequencies}).
	Basu and Mj use a divisibility condition on word lengths to obtain mixing \cite[Lemma~3.20]{BasuMj2022}.
	We obtain existence for all word lengths using the finite-state Markov chain ergodic theorem, without assuming mixing.

\subsection*{Examples}

	A finitely generated group is \emph{geodesically completable} if every finite symmetric generating set can be enlarged to one admitting a geodesic biautomatic structure.
	Antol\'in and Ciobanu show that this property passes from peripheral subgroups to the relatively hyperbolic group \cite[Theorem~1.5]{AntolinCiobanu2016}, and the resulting structures supply the combings needed here \cite[Proposition~2.4]{GekhtmanTaylorTiozzo2020}.
	In particular, the results apply to non-elementary relatively hyperbolic groups with virtually abelian peripherals after enlarging the generating set if necessary \cite[Proposition~10.1]{AntolinCiobanu2016}.
	The example \(\mathbb Z*\mathbb Z^2\), is hyperbolic relative to \(\mathbb Z^2\) but not word-hyperbolic.
    This shows our results go beyond the hyperbolic setting.
	Appendix~\ref{app:worked-examples} works out its geodesic combing, Perron--Frobenius data, and ordered frequencies explicitly.

\subsection*{Organization of the paper}
The rest of the paper is organized as follows.
Section~\ref{sec:geodesic-combings} introduces geodesic combings and the peripheral hypotheses that guarantee their existence in the relatively hyperbolic setting.
Section~\ref{sec:pf-markov} then constructs the Markov measures following Calegari--Fujiwara \cite{CalegariFujiwara2010}, using Yang's growth theorem \cite{Yang2022} to verify the required growth of paths.
We include the language and Markov-chain background for readers coming from geometry, and recall the geometric definitions for readers coming from those subjects.
We also give the measure construction explicitly, since the choice of initial state and normalization matter in the boundary comparison.

Section~\ref{sec:markov-ps-comparison} shows the existence of a continuous surjection between the space of paths in the geodesic combing and the Bowditch compactification, and proves the measure identity by lifting weighted counting measures to finite paths.
Section~\ref{sec:frequency} proves existence of ordered frequencies, first following the period-restricted argument of Basu and Mj, and then removing the restriction using the ergodic theorem.
Readers interested mainly in word frequencies can follow this Markov-chain argument after Section~\ref{sec:geodesic-combings}; Section~\ref{sec:markov-ps-comparison} is needed to interpret the conclusion on the Bowditch boundary.
Appendix~\ref{app:worked-examples} supplies an explicit combing, growth calculations, and ordered frequencies for \(\mathbb Z*\mathbb Z^2\).

\subsection*{Acknowledgments}
The author thanks Ilya Gekhtman and Yago Antolin for helpful correspondence concerning geodesic combings and geodesic biautomatic structures in relatively hyperbolic groups, Mahan Mj for helpful discussions about ordered frequency and for suggesting the relatively hyperbolic setting, and Yulan Qing and Milap Rajgor for discussing many aspects of this project throughout its development.

\subsection*{Use of AI}
We acknowledge the use of OpenAI's Codex (primarily using the model GPT 5.6 Sol as well as GPT 6 Astra), a generative artificial intelligence tool, during the preparation of this manuscript to assist with organizing the exposition, revising prose, drafting \LaTeX, proof writing, discussing mathematical arguments, and numerical simulation.
We review and revise all AI-generated suggestions and assume full responsibility for the final manuscript.
We attempt to be as transparent as possible and we discuss how AI is used throughout the manuscript. 

\section{Geodesic combings in relatively hyperbolic groups}\label{sec:geodesic-combings}

The goal of this section is to define geodesic combings and identify a class of relatively hyperbolic groups that admit them.
We then use the growth properties of nonelementary relatively hyperbolic groups to show that every such geodesic combing is almost semisimple in the sense of Definition~\ref{def:almost-semisimple-automaton}.
Gekhtman--Taylor--Tiozzo treat this almost semisimplicity as a consequence of pure exponential growth and the defining bijection of a geodesic combing \cite[discussion following Theorem~9.2]{GekhtmanTaylorTiozzo2020}.
For completeness, Lemma~\ref{lem:combing-almost-semisimple} makes this deduction explicit, including the passage from pure exponential growth of balls to pure exponential growth of spheres (Lemma~\ref{lem:pure-sphere-growth}).

Once almost semisimplicity has been established, we largely follow the Perron--Frobenius and Markov-chain construction of Calegari--Fujiwara \cite{CalegariFujiwara2010}.
They first construct a stationary probability distribution on the vertices of an almost semisimple graph structure and then use the transition probabilities to obtain a shift-invariant Markov measure on the space of infinite paths.
We also construct a \emph{based Markov measure} \(\mu^0\) on the space \(\mathcal P_\infty^0\) of infinite paths beginning at the initial vertex \(v_0\).
Although some of the relevant arguments could be cited directly from \cite{CalegariFujiwara2010}, we include the details needed for our normalization and for the distinction between the stationary and based measures.
In particular, Lemma~\ref{lem:based-markov-measure} gives an explicit formula for the measure of every based cylinder.

Together, Lemmas~\ref{lem:combing-almost-semisimple} and \ref{lem:based-markov-measure} produce a canonical probability measure on the infinite based paths in a geodesic combing of a nonelementary relatively hyperbolic group.
In the next section, we extend evaluation continuously from finite based paths to their compactification and show that the pushforward of \(\mu^0\) is a Patterson--Sullivan measure.

\subsection{Geometry review and conventions}

We begin by fixing how words describe routes in a Cayley graph.
The finite directed graph introduced next serves a different purpose: it records which spellings we choose.
Keeping these two graphs separate will help us pass from shortest representatives to statistics of paths.
\par

Throughout this manuscript, all groups are assumed to be countable and finitely generated, every generating set is assumed to be finite and symmetric, and no generating set contains the identity.
For a generating set \(S=S^{-1}\), let \(\Cay(G,S)\) denote the Cayley graph with every edge assigned length one, let \(d_S\) denote its \emph{edge} or \emph{word metric}, and let \(|g|_S=d_S(\id,g)\).
\emph{We emphasize that we are considering \(\Cay(G,S)\) as an undirected graph.}
All geodesics in the Cayley graph are assumed to be parametrized by arc length, thus travel at unit speed.
A \emph{geodesic ray} is an isometric embedding \(\gamma:[0,\infty)\to \Cay(G,S)\).



\subsection{Regular languages, graph structures, and geodesic combings}
We now review some basic facts about regular languages, automatic groups, and graph structures.
We refer the reader to \cite{EpsteinEtAl1992} and \cite{HoltReesRover2017} for comprehensive background information on the subject.


\begin{definition} 
    An \emph{alphabet} is a finite set of elements called \emph{letters}.
    A \emph{word} over an alphabet is a finite string of ordered letters.
    A \emph{language} is a collection of words. Let \(S:=\{s_1,s_2,\dots s_m\}\) be an alphabet.
    We use the notation \(\ell(w)\) to denote the length of the word \(w\in S^*\).
    Let \(S^*\) denote the free monoid of all possible words over \(S\), including the empty word.
    Let \(\mathcal{L} \subseteq S^*\) denote a language over \(S\).
    \(\mathcal{L}\) is said to be \emph{prefix-closed} if for any \(w=uv \in \mathcal{L}\), where \(u,v,w \in S^*\), we have the prefix \(u \in \mathcal{L}\). 
\end{definition}

\begin{definition}\label{def:automaton} Let \(S\) be an alphabet.
    An \emph{\(S\)-labeled graph structure} (or simply graph structure) is a finite directed (multi)graph with a distinguished initial vertex \(v_0\) with no incoming edges such that every vertex is accessible from \(v_0\) and distinct edges leaving the same vertex have distinct labels from \(S\).
    An \(S\)-labeled graph structure \(\mathcal G\) is said to \emph{accept} a word \(w \in S^*\) if there exists a direct path in \(\mathcal G\) starting at \(v_0\) such that reading the labels in order along the path exactly spells \(w\).
    A language \(\mathcal L \subseteq S^*\) is called \emph{regular} if there exists graph structure that accepts exactly every word \(w\in \mathcal L\); such graph structure is said to \emph{parametrize} \(\mathcal{L}\).
	The underlying directed graph of \(\mathcal G\) is also referred to as a \emph{topological Markov chain} and its vertices are referred to as \emph{states}.
\end{definition}

In general, there are infinitely many graph structures that parameterize the same language, and depending on the application, the choice may not be important.
However, we are interested in statistics concerning the edge labels of a given path in the graph structure. 
The definition of a graph structure allows for \emph{parallel edges}; i.e. edges that share the same head and tail vertices.
Allowing parallel edges creates ambiguity between vertex paths and labeled paths.
Indeed, if two edges have the same head and tail vertices, then a sequence of vertices may not uniquely determine a labeled edge path.
The following lemma shows that this ambiguity can be removed by choosing a new graph structure which accepts exactly the same language, but has no parallel edges.
After making this replacement, every labeled edge path is uniquely determined by its sequence of vertices.

\begin{lemma}\label{lem:parallel-free}
    Let \(\mathcal G\) be an \(S\)-labeled graph structure with initial vertex \(v_0\).
    Then there exists an \(S\)-labeled graph structure \(\mathcal G'\) such that \(\mathcal L_{\mathcal G'}=\mathcal L_{\mathcal G}\) and \(\mathcal G'\) has no parallel edges.
\end{lemma}
\begin{proof}[Proof of Lemma~\ref{lem:parallel-free}]\label{app:parallel-free}
    For a directed edge \(e\in E(\mathcal G)\), let \(t(e)\) and \(h(e)\) denote its tail and head vertex, respectively.
    For each edge \(e\in E(\mathcal G)\), introduce a vertex \(q_e\), and introduce one additional vertex \(v_0'\).
    Define
    \[
    V(\mathcal G')=\{v_0'\}\sqcup\{q_e:e\in E(\mathcal G)\},
    \]
    where \(v_0'\) is the distinguished initial vertex.
    For each edge \(e\in E(\mathcal G)\) with \(t(e)=v_0\), add one directed edge
    \[
    v_0'\longrightarrow q_e
    \]
    carrying the same label as \(e\).
    For each ordered pair of edges \(e,f\in E(\mathcal G)\) satisfying \(h(e)=t(f)\), add one directed edge
    \[
    q_e\longrightarrow q_f
    \]
    carrying the same label as \(f\).
    Observe that \(\mathcal G\) is finite, so is \(\mathcal G'\) with this construction.
    
    Let \(p\) be a path in \(\mathcal G\) beginning at \(v_0\), let \(e_1,\ldots,e_n\) be its successive edges, and let \(s_i\in S\) be the label of \(e_i\).
    Writing \(v_{i_0}=v_0\) and \(v_{i_j}=h(e_j)\) for \(1\leq j\leq n\), we may display \(p\) as
    \[
    p:\quad v_{i_0}\xrightarrow{s_1}v_{i_1}\xrightarrow{s_2}v_{i_2}\longrightarrow\cdots\xrightarrow{s_n}v_{i_n}.
    \]
    Associate to \(p\) the path \(p'\) displayed as
    \[
    p':\quad v_0'\xrightarrow{s_1}q_{e_1}\xrightarrow{s_2}q_{e_2}\longrightarrow\cdots\xrightarrow{s_n}q_{e_n}
    \]
    in \(\mathcal G'\).
    Conversely, every path in \(\mathcal G'\) beginning at \(v_0'\) has this form and determines a unique path in \(\mathcal G\) beginning at \(v_0\).
    Together with the correspondence between the two empty paths, this gives a length-preserving bijection between the based paths in \(\mathcal G\) and those in \(\mathcal G'\).
    By construction, corresponding paths have the same label word, and therefore
    \( \mathcal L_{\mathcal G'}=\mathcal L_{\mathcal G}.\)
    
    Every vertex \(q_e\) is accessible from \(v_0'\), because every \(t(e)\) is accessible from \(v_0\) in \(\mathcal G\).
    Moreover, distinct edges leaving \(v_0'\) have distinct labels because they correspond to distinct edges leaving \(v_0\), and distinct edges leaving \(q_e\) have distinct labels because they correspond to distinct edges leaving \(h(e)\).
    
    Finally, there is exactly one edge from \(v_0'\) to \(q_e\) whenever \(t(e)=v_0\), and there is at most one edge from \(q_e\) to \(q_f\), determined by the ordered pair \((e,f)\).
    Therefore, \(\mathcal G'\) has no parallel edges.
\end{proof}

\begin{ainotice}
    Lemma \ref{lem:parallel-free} was suggested and proven by AI. 
    The necessity of the result is mostly a consequence of our conventions. 
    We could have specified that we are considering only graph structures with no parallel edges, but the authors of \cite{CalegariFujiwara2010} and \cite{GekhtmanTaylorTiozzo2020} allow for the possibility of parallel edges, and we attempt to follow their conventions.
    The authors of \cite{BasuMj2022} implicitly assume there are no parallel edges by their construction of the adjacency matrix.
\end{ainotice}

In view of Lemma~\ref{lem:parallel-free}, we henceforth assume without loss of generality that every graph structure has no parallel edges.
Therefore, every labeled path in \(\mathcal G\) is uniquely determined by its sequence of vertices.
    
If \(S\) is a finite generating set of a group \(G\), then there exists a natural evaluation map \(\ev: S^*\to G\) that sends a word \(w\in S^*\) to the group element \(g = \ev(w)\) it represents.
The empty word is assumed to be sent to the identity \(\id \in G\).  
Let \(w=s_1\cdots s_n\in S^*\).
The word \(w\) determines an edge path \(\gamma_w:[0,n]\to\Cay(G,S)\) beginning at \(\id\), defined by \(\gamma_w(j)=\ev(s_1\cdots s_j)\) for \(0\leq j\leq n\), where \(\gamma_w(0)=\id\).
Thus, \(\gamma_w\) is parametrized to have unit speed on \([0,n]\).

\begin{definition}\label{def:paths_and_languages}
    Suppose that \(S\) is a finite symmetric generating set of a group \(G\), and let \(\mathcal G\) be an \(S\)-labeled graph structure with initial vertex \(v_0\).
    Let \(\mathcal P^0\) be the set of finite directed paths in \(\mathcal G\) beginning at \(v_0\), which we call \emph{based paths}, and let \(\mathcal P_n^0\subseteq\mathcal P^0\) be the set of based paths of length \(n\).
    Reading the edge labels along based paths and concatenating the labels on the right determines a \emph{regular prefix-closed language} \(\mathcal L\subseteq S^*\).
    Since distinct edges leaving the same vertex have distinct labels, every word in \(\mathcal L\) determines a unique based path, and hence there is a natural bijection \(\mathcal P^0\longleftrightarrow\mathcal L\).
    The evaluation map \(\ev:S^*\to G\) therefore induces an evaluation map \( \ev:\mathcal P^0\longrightarrow G\)
    sending each based path to the group element represented by its label word.
    The path of length zero consisting only of \(v_0\) is sent to \(\id\).
\end{definition}

\begin{definition}\label{def:geodesic-combing}
    We call a graph structure \emph{surjective} (resp. \emph{injective}) if \(\ev:\mathcal P^0\to G\) is surjective (resp. injective).
    A graph structure is called \emph{bijective} if it is both  surjective and injective.
    A graph structure is called \emph{geodesic} if it is surjective, and for each path \(y\in \mathcal{P}^0\), \(|y| = |\ev(y)|_S\).
    Finally, a graph structure is called a \emph{geodesic combing} (or simply combing) if it is bijective and geodesic.
    The accepted language of a combing is called a \emph{geodesic normal form language}.
\end{definition}

\begin{remark}
    There seems to be little consensus on what should be called a geodesic combing.
    Some definitions are equivalent, but not all.
    See \cite{Rees1998} for some discussion.
    We choose to follow closely to the conventions of \cite{GekhtmanTaylorTiozzo2020}.
\end{remark}

\subsection{Geodesic combings and relatively hyperbolic groups}


The definitions above specify the kind of language we want: one shortest representative for each group element, recognized by a finite graph.
We now ask when relatively hyperbolic groups admit such a language.
After recalling relative hyperbolicity, we give a sufficient condition on the peripheral subgroups, allowing the generating set to be enlarged.


\begin{definition}\cite[Definition 2.2]{KimOhTranchida2022}
    Let \(G\) be a finitely generated group with finite symmetric generating set \(S\) and \(\mathcal H\) a finite collection of proper finitely generated infinite subgroups of \(G\). 
    The \emph{coned--off} Cayley graph, \(\widehat{\Cay}(G,S,\mathcal H)\) is defined as follows: starting with the original Cayley graph \(\Cay(G,S)\), for each left coset \(gH\) with \(g\in G\) and \(H\in \mathcal H\), add a new vertex \(v_{gH}\) called a \emph{cone vertex} to \(\Cay(G,S)\) and connect \(v_{gH}\) by an edge to each vertex in \(gH\). 
    We then say that \emph{(\(G\) is hyperbolic relative to \(\mathcal H\)} if \(\widehat{\Cay}(G,S,\mathcal H)\) is hyperbolic and \emph{fine}; that is, for every \(n\in \mathbb N\), each edge \(e\in \widehat{\Cay}(G,S,\mathcal H)\) is contained in at most finitely many simple cycles of length \(n\).
    We call \((G,\mathcal H)\) a \emph{relatively hyperbolic pair} or a \emph{relatively hyperbolic group}.
\end{definition}

There are a myriad of equivalent definitions one could use for relatively hyperbolic groups.
We choose these definitions as they seem to be the most easily applicable for our purposes.
For further reading on relatively hyperbolic groups and their properties, see \cite{Bowditch2012RelativelyHyperbolic, Farb98, Gerasimov2012FloydMaps, GrovesManning2008, Gromov1987HyperbolicGroups, Hruska2010}.
We now describe the class of relatively hyperbolic groups that we will study.

\begin{definition}
    Let \(G\) be a group with finite generating set \(S\).
    A \emph{geodesic biautomatic structure} for \((G,S)\) is a regular language \(\mathcal L\subseteq S^*\) that evaluates surjectively to \(G\), consists of geodesic words, and satisfies the uniform synchronous fellow-traveler property for translated representative paths whose corresponding initial points and corresponding terminal points are distance at most one apart.
	Synchronous fellow traveling means that the paths remain within one uniform distance at equal times, with each path held at its terminal point after it ends; see \cite[Definition~7.3]{AntolinCiobanu2016}.
\end{definition}

\begin{definition}\label{def:geodesically-completable}
    A finitely generated group \(G\) is \emph{geodesically completable} if every finite generating set \(S\) of \(G\) is contained in a finite generating set \(S'\) for which there is a geodesic biautomatic structure for \((G,S')\).
\end{definition}


\begin{theorem}\cite[Theorem 1.5]{AntolinCiobanu2016}\label{thm:ac-geodesic-completability}
    Let \((G,\mathcal H)\) be a relatively hyperbolic pair.
    If every \(H\in\mathcal H\) is geodesically completable, then \(G\) is geodesically completable.
\end{theorem}

In view of this theorem, let us call a finite generating set \(S\) \emph{admissible} if there is a geodesic biautomatic structure for \((G,S)\).
Combining this definition with the above theorem, we immediately come to the following conclusion.

\begin{proposition}
    Let \((G,\mathcal H)\) be a relatively hyperbolic pair such that every \(H\in\mathcal H\) is geodesically completable.
    Then every finite generating set \(S\) of \(G\) is contained in an admissible finite generating set \(S'\).
\end{proposition}

As it turns out, by \cite[Theorem 5.2.7]{HoltReesRover2017}, if there exists a geodesic biautomatic structure for \((G,S)\), then there exists a geodesic combing \(\mathcal G\) of \(\Cay(G,S)\). 
Thus, we arrive at the following conclusion.

\begin{proposition}\cite[Proposition 2.4]{GekhtmanTaylorTiozzo2020}
    Let \((G,\mathcal H)\) be a relatively hyperbolic pair such that every \(H\in\mathcal H\) is geodesically completable.
    Then there exists some admissible generating set \(S\) of \(G\) such that there is a geodesic biautomatic structure for \((G,S)\).
    Moreover, there exists a geodesic combing \(\mathcal G\) of \(\Cay(G,S)\).
\end{proposition}

\cite[Corollary 1.9 and Proposition 10.1]{AntolinCiobanu2016} show that the relatively hyperbolic pair \((G,\mathcal H)\) is geodesically completable if the peripheral subgroups, \(\mathcal H\), are virtually abelian. So the following corollary is immediate.

\begin{corollary}
    Let \((G,\mathcal H)\) be a relatively hyperbolic pair such that every \(H\in \mathcal H\) is virtually abelian.
    Then there exists some admissible generating set \(S\) of \(G\) such that there is a geodesic biautomatic structure for \((G,S)\).
    Moreover, there exists a geodesic combing \(\mathcal G\) of \(\Cay(G,S)\).
\end{corollary}

This corollary shows that the class of relatively hyperbolic groups that admit geodesic combings, at the very least, contains some groups that are not hyperbolic.
Many \(\operatorname{CAT}(0)\) groups, mapping class groups, and Artin groups are hyperbolic relative to a family of virtually abelian subgroups.
To give a concrete example, consider the group \(G=\mathbb Z * \mathbb Z^2\).
This group is clearly not hyperbolic as it contains \(\mathbb Z^2\) as a subgroup, but \(G\) is hyperbolic relative to \(\mathbb Z^2\), which is (virtually) abelian.

\begin{remark}
    One obvious shortcoming of these results is that we only state the existence of admissible generating sets.
    The authors \cite{AntolinCiobanu2016} remark that these generating sets are, in fact, computable, but provide no algorithm for doing so.
    We refer the reader to \cite{AntolinCiobanu2016} who, in turn, refer to \cite{Dahmani08} for more information.
    Henceforth, we assume all generating sets are admissible, and thus admit a geodesic combing.
\end{remark}

\begin{remark}
    A second shortcoming of these results, at least for our purposes, is that the authors of \cite{AntolinCiobanu2016} do not assume generating sets to be symmetric.
    On the other hand, the authors of \cite{GekhtmanTaylorTiozzo2020} assume all generating sets are symmetric and that the above results hold under this assumption.
    Patterson--Sullivan measures are defined through a word metric on a finitely generated group.
    Without assuming a symmetric generating set, the induced Cayley graph becomes a directed graph, and the induced word ``metric'' may not be symmetric.
    After corresponding with Gekhtman and Antolin, we (meaning the author of this manuscript) conclude that \mbox{\cite[Theorem 7.6]{AntolinCiobanu2016}} subtlety implies that admissible generating sets can be assumed to be symmetric, but we do not provide a proof.
    Henceforth, we assume that all generating sets are symmetric and that the above results hold under this assumption.
	
\end{remark}

We now have examples of groups with the required combings.
The next task is to turn counts of finite paths into a probability law on infinite paths.
We first carry out this construction for a finite graph satisfying the growth condition called almost semisimplicity, and then verify that condition for our combings.


\section{Perron--Frobenius theory, almost semisimple graph structures, and Markov measures}\label{sec:pf-markov}

Let \(\mathcal G\) be a graph structure and let \(M\) be the adjacency matrix of \(\mathcal G\), where \(M_{ij} =1\) if there is a directed edge from \(v_i\) to \(v_j\) and 0 otherwise.
In general, \(M_{ij}\) is the number of directed edges from \(v_i\) to \(v_j\), but we are assuming we have chosen a parallel edge free graph structure due to Lemma \ref{lem:parallel-free}.
Let \(\mathbb R^{V(\mathcal G)}\) denote the real vector space with standard basis \(\{\mathbf e_i:v_i\in V(\mathcal G)\}\) indexed by the vertices of \(\mathcal G\).
We regard the adjacency matrix \(M\) as a linear operator on \(\mathbb R^{V(\mathcal G)}\).
Let
\[
	\mathbf 1:=\sum_{v_i\in V(\mathcal G)}\mathbf e_i
\]
denote the vector whose coordinates are all equal to \(1\).
Let \(\lambda\) denote the spectral radius of \(M\).
By Perron-Frobenius, \(\lambda\) is in fact the largest eigenvalue.

\begin{lemma}\cite[Lemma~4.2]{CalegariFujiwara2010}\label{lem:CF-4.2}
	For any vertices \(v_i,v_j\in V(\mathcal G)\), the number of directed paths in \(\mathcal G\) of length \(n\) from \(v_i\) to \(v_j\) is
	\[
		\mathbf e_i^TM^n\mathbf e_j=(M^n)_{ij}.
	\]
\end{lemma}
\begin{proposition}\label{prop:number_of_paths}
	For every vertex \(v_i\), the number of directed paths in \(\mathcal G\) of length \(n\) beginning at \(v_i\) is
	\[
		\sum_j(M^n)_{ij}=\mathbf e_i^TM^n\mathbf 1.
	\]
\end{proposition}
\begin{proof}
    By Lemma \ref{lem:CF-4.2}, the number of directed paths from \(v_i\) to each \(v_j\) is \((M^n)_{ij}\). Summing over all vertices \(v_j\) yields the desired result.
\end{proof}

Let \(\mathcal G\) be a graph structure and let \(\mathcal P\) denote the set of all finite directed paths in \(\mathcal G\), without a condition on the initial vertex, and let \(\mathcal P_n\) denote the paths of length \(n\).
Restriction to prefixes gives inverse systems whose limits \(\mathcal P_\infty^0\) and \(\mathcal P_\infty\) are the spaces of infinite paths beginning at \(v_0\) and at arbitrary vertices, respectively.
Both spaces carry the prefix topology, and a basic clopen set is the cylinder
\[
[y]:=\{\sigma\in\mathcal P_\infty^0:\text{\(\sigma\) begins with \(y\)}\} \quad \left(\text{resp. }[x]:=\{\sigma\in\mathcal P_\infty:\text{\(\sigma\) begins with \(x\)}\}\right)
\]
determined by a finite based path \(y\in\mathcal P^0\), (resp. finite path \(x\in \mathcal P\)).
The compactified based path space is the compact metrizable space \(\overline{\mathcal P}^{\,0}:=\mathcal P^0\sqcup\mathcal P_\infty^0\), where every finite path is isolated and a sequence converges to an infinite path precisely when its common prefix length tends to infinity.


\begin{definition}\label{def:almost-semisimple-automaton}
A graph structure is said to be \emph{almost semisimple} if every vertex is accessible from \(v_0\) and there are constants \(K\geq 1\) and \(\lambda>1\) such that
\[
K^{-1}\lambda^n\leq |\mathcal P_n^0|\leq K\lambda^n
\]
for every \(n\geq 1\).
\end{definition}

Let \(\mathbf e_0\) be the coordinate vector corresponding to \(v_0\).
The following formulation collects the Perron--Frobenius consequences of almost semisimplicity that will be used below.

\begin{lemma}[\cite{CalegariFujiwara2010}, Lemma 4.5, 4.6]\label{lem:cesaro-pf-projections}
Let \(\mathcal{G}\) be an almost semisimple graph structure. Then the limits
\[
r:=\lim_{n\to\infty}\frac{1}{n}\sum_{i=0}^{n}\lambda^{-i}M^i\mathbf 1
\qquad\text{and}\qquad
\ell:=\lim_{n\to\infty}\frac{1}{n}\sum_{i=0}^{n}\lambda^{-i}(M^T)^i\mathbf e_0
\]
exist, are nonzero and nonnegative, and satisfy
\[
Mr=\lambda r
\qquad\text{and}\qquad
M^T\ell=\lambda\ell.
\]
They are the projections of \(\mathbf 1\) and \(\mathbf e_0\) to the right and left \(\lambda\)-eigenspaces, respectively.
\end{lemma}

Assume throughout the following construction that the graph structure is almost semisimple.
Define a matrix \(N\) by 
\begin{equation}\label{eq:state-transition-probability}
N_{ij}:=\frac{M_{ij}r_j}{\lambda r_i}
\end{equation}
whenever \(r_i>0\), otherwise define \(N_{ii} = 1\) and \(N_{ij} = 0\) for \(i\neq j\).
Define a probability measure \(\mu\) on the vertices of \(\mathcal G\) by
\[
    \mu_i=\frac{r_i\ell_i}{\sum_i r_i\ell_i}.
\]
\cite[Lemma 4.9]{CalegariFujiwara2010} shows that \(N\) is stochastic and preserves the vertex distribution \((\mu_i)\); moreover, \(N_{ij}\) is the probability of moving from \(v_i\) to \(v_j\).
For each directed path \(x=v_{i_0}v_{i_1}\cdots v_{i_n}\in\mathcal P_n\) in \(\mathcal{G}\), define the measure (which we also denote by \(\mu\))
\[
    \mu([x]):=\mu_{i_0}N_{i_0i_1}N_{i_1i_2}\cdots N_{i_{n-1}i_n},
\]
where \([x]:=\{\sigma\in\mathcal P_\infty:\text{\(\sigma\) begins with \(x\)}\}\).
As cylinders determine a clopen basis for the topology of \(\mathcal P\) these assignments imply \(\mu\) is a Borel probability measure on \(\mathcal P_\infty\).
Let \(\mathsf S:\mathcal P_\infty\to\mathcal P_\infty\) denote the shift map obtained by deleting the initial vertex of an infinite path.
Since the vertex distribution \((\mu_i)\) is stationary under \(N\), the Markov measure \(\mu\) is \(\mathsf S\)-invariant; equivalently,
\[
    \mu\bigl(\mathsf S^{-1}(A)\bigr)=\mu(A)
\]
for every Borel set \(A\subseteq\mathcal P_\infty\). We thus call \(\mu\) a \emph{stationary measure}.


The stationary law chooses a starting state according to its stationary distribution.
For rays issuing from the identity, we instead need every path to start at the automaton's initial vertex.
We keep the same transition probabilities and change only this choice of starting state.

For each finite based directed path \(y = v_0 v_{i_1}\dots v_{i_n} \in \mathcal P^0_n\) in \(\mathcal G\), assign to its cylinder the mass 
\[
    m([y]) = N_{0i_1}N_{i_1i_2}\cdots N_{i_{n-1}i_n}
\]

\begin{lemma}[Based Markov measure]\label{lem:based-markov-measure}
    The assignment \(m\) extends uniquely to a Borel probability measure \(\mu^0\) on \(\mathcal P_\infty^0\). 
    Moreover, for every \(y=v_0v_{i_1}\cdots v_{i_n}\in\mathcal P^0,\) we have
    \[
        \mu^0([y])=\frac{\lambda^{-|y|}r_{v_{i_n}}}{r_{v_0}}.
    \]
\end{lemma}

\begin{proof}
For the length-zero path, observe
\[
m([v_0]) =m(\mathcal P^0_\infty)=\lambda^0r_{v_0} / r_{v_0}=1.
\]
Now let \( y=v_0v_{i_1}\cdots v_{i_n}\in\mathcal P^0.\)
Using the definition of \(N\) from
\eqref{eq:state-transition-probability}, we obtain
\begin{align*}
m([y])
&=
N_{0i_1}N_{i_1i_2}\cdots N_{i_{n-1}i_n}\\
&=
\frac{M_{0i_1}r_{v_{i_1}}}{\lambda r_{v_0}}
\frac{M_{i_1i_2}r_{v_{i_2}}}{\lambda r_{v_{i_1}}}
\cdots
\frac{M_{i_{n-1}i_n}r_{v_{i_n}}}
     {\lambda r_{v_{i_{n-1}}}}\\
&=
\left(\prod_{k=1}^n M_{i_{k-1}i_k}\right)
\frac{r_{v_{i_n}}}{\lambda^n r_{v_0}}\\
&=
\frac{\lambda^{-|y|}r_{v_{i_n}}}{r_{v_0}}.
\end{align*}
The last equality holds because we assume the graph structure has no parallel edges, and hence \(M_{i_{k-1}i_k}=1\) for every transition occurring in \(y\).

It remains to verify that these cylinder masses extend to a Borel probability measure on \(\mathcal P^0_\infty\).
Let \(y\in\mathcal P_k^0\) terminate at \(v_i\).
The cylinder \([y]\) is the disjoint union of the cylinders corresponding to the one-edge extensions of \(y\):
\[
    [y]= \bigsqcup_{\substack{ z\in\mathcal P_{k+1}^0\\ z\text{ extends }y}} [z].
\]

    Moreover,
    \begin{align*}
    	\sum_{\substack{z\in\mathcal P_{k+1}^0\\z\text{ extends }y}}m([z])
    	&=\frac{\lambda^{-(k+1)}}{r_{v_0}}\sum_j M_{ij}r_{v_j}\\
    	&=\frac{\lambda^{-(k+1)}(Mr)_i}{r_{v_0}}\\
    	&=\frac{\lambda^{-k}r_{v_i}}{r_{v_0}}\\
    	&=m([y]).
    \end{align*}
    The third equality follows from \(Mr=\lambda r\), as established in Lemma~\ref{lem:cesaro-pf-projections}.
    Thus, when \([y]\) is decomposed into its pairwise disjoint one-edge subcylinders, the mass assigned to \([y]\) is the sum of the masses assigned to those subcylinders.
    Iterating this identity gives the corresponding additivity for the partition of \([y]\) into cylinders determined by extensions of any fixed length.
    
    Equivalently, the probability distributions on the finite path spaces \(\mathcal P_k^0\) defined by
    \[
    	p_k(y):=m([y])
    \]
    are compatible under the prefix maps \(\mathcal P_{k+1}^0\to\mathcal P_k^0\).
    The Kolmogorov extension theorem therefore gives a unique probability measure \(\mu^0\) on \(\mathcal P_\infty^0\) satisfying
    \[
    	\mu^0([y])=m([y])
    \]
    for every finite based path \(y\).
    Since the cylinders form a clopen basis for \(\mathcal P_\infty^0\), they generate its Borel \(\sigma\)-algebra, and hence \(\mu^0\) is uniquely determined as a Borel probability measure.
\end{proof}

We call \(\mu^0\) the \emph{based Markov measure}, since it is supported on the infinite paths beginning at \(v_0\).
We also regard \(\mu^0\) as a measure on \(\mathcal P_\infty\), assigning zero mass to paths not beginning at \(v_0\).
Since \(\mathcal G\) is a geodesic combing, \(v_0\) has no incoming edges; thus, \(\mathcal P^0_\infty\) and \(\mathsf S(\mathcal P^0_\infty)\) are disjoint, and so \(\mu^0\) is not shift-invariant.
Nevertheless, its Cesàro shift averages converge to the stationary Markov measure:
\[
	\lim_{n\to\infty}\frac{1}{n}\sum_{i=1}^n(\mathsf S^i)_*\mu^0 = \mu.
\]
See \cite[\S 4.3]{CalegariFujiwara2010}.

\subsection{Growth and almost semisimplicity}


The probability laws are now constructed, provided the graph has the required exponential path growth.
We return to the group to check this condition for geodesic combings.
Yang's theorem \cite[Theorem~1.9]{Yang2022} counts elements within a given distance; we need counts at an exact distance, because these correspond to based paths of a fixed length.


For \(n\geq 0\), let
\[
G_n:=\{g\in G:|g|_S=n\} \qquad\text{and}\qquad B_n:=\{g\in G:|g|_S\leq n\}.
\]

\begin{theorem}\cite[Theorem 1.9]{Yang2022}\label{thm:yang-pure-ball-growth}
Let \((G,\mathcal H)\) be any nonelementary relatively hyperbolic pair. Then there exists constants \(\lambda >1,~C\geq 1\) such that
\[
C^{-1}\lambda^ n\leq |B_n|\leq C\lambda^n
\]
for every \(n\geq 1\).
\end{theorem}

This growth behavior is referred to as \emph{purely exponential growth}. Using this theorem we deduce that the growth of \(G_n\) must also be purely exponential via the following lemma.

\begin{lemma}\label{lem:pure-sphere-growth}
Let \((G,\mathcal H)\) be any nonelementary relatively hyperbolic pair. Then there exists constants \(\lambda >1,~K\geq 1\) such that
\[
K^{-1}\lambda^n\leq |G_n|\leq K\lambda^n
\]
for every \(n\geq 1\).
\end{lemma}

\begin{proof}
By Theorem~\ref{thm:yang-pure-ball-growth},
The upper bound is immediate:
\[
|G_n|\leq |B_n|\leq C\lambda^n.
\]
Choose \(N\geq 1\) sufficiently large that
\[
C^{-1}\lambda^N-C>0.
\]
Since the annulus \(B_{n+N}\setminus B_n\) is the disjoint union of the spheres \(G_{n+1},\ldots,G_{n+N}\), we have
\begin{align*}
\sum_{i=1}^N|G_{n+i}|&=|B_{n+N}|-|B_n|,\\
&\geq C^{-1}\lambda^{n+N}-C\lambda^n,\\
&=(C^{-1}\lambda^N-C)\lambda^n.
\end{align*}
It follows that there is some \(j\in\{1,\ldots,N\}\) such that
\begin{equation}\label{eq:large-nearby-sphere}
|G_{n+j}|\geq\frac{C^{-1}\lambda^N-C}{N}\lambda^n.
\end{equation}
For each \(h\in G_{n+j}\), choose a geodesic word \(w_h\in S^*\) representing \(h\), and define the prefix map \(p:G_{n+j}\to G_n\) by \(p(h)=\ev((w_h)_n)\), where \((w_h)_n\) is the prefix of \(w_h\) of length \(n\).
Every prefix of a geodesic word is geodesic, so \(p(h)\in G_n\).
Fix \(g\in G_n\).
If \(h\in p^{-1}(g)\), then \(w_h=(w_h)_n u\) for some \(u\in S^j\), and hence \(h=g\ev(u)\).
There are at most \(|S|^j\) possible words \(u\in S^j\), so
\[
|p^{-1}(g)|\leq |S|^j\leq |S|^N.
\]
Consequently,
\[
|G_{n+j}|\leq |S|^N|G_n|.
\]
Combining this estimate with \eqref{eq:large-nearby-sphere} gives
\[
|G_n|\geq\frac{C^{-1}\lambda^N-C}{N|S|^N}\lambda^n.
\]

Letting
\[
K:=\max\left\{C,\frac{N|S|^N}{C^{-1}\lambda^N-C}\right\}
\]
gives
\[
K^{-1}\lambda^n\leq |G_n|\leq K\lambda^n
\]
for every \(n\geq 1\).

\end{proof}

Since the combing is bijective and geodesic, evaluation restricts to a bijection \(\ev:\mathcal P_n^0\longrightarrow G_n.\)
Combining this with Proposition \ref{prop:number_of_paths} gives,
\begin{equation}\label{eq:based-path-count}
|\mathcal P_n^0|=|G_n|=\sum_j(M^n)_{0j} = \mathbf{e_0}^TM^n\mathbf{1},
\end{equation}
where \(\mathbf e_0\) is the coordinate vector corresponding to \(v_0\).

\begin{lemma}\label{lem:combing-almost-semisimple}
	Let \((G,\mathcal H)\) be a non-elementary relatively hyperbolic pair with admissible finite symmetric generating set \(S\) and let \(\mathcal G\) be a geodesic combing of \(\Cay(G,S)\).
	Then \(\mathcal G\) is almost semisimple.
\end{lemma}

\begin{proof}
Every vertex of \(\mathcal G\) is accessible from \(v_0\) by Definition~\ref{def:automaton}.
The lemma then follows by combining \eqref{eq:based-path-count} and Lemma~\ref{lem:pure-sphere-growth}.
\end{proof}

\section{Extended evaluation and Patterson--Sullivan measure}\label{sec:markov-ps-comparison}
Fix a non-elementary relatively hyperbolic pair \((G,\mathcal H)\), an admissible finite symmetric generating set \(S\), and a geodesic combing \(\mathcal G\) of \(\Cay(G,S)\).
Recall that the pair \((G,\mathcal H)\) admits a natural geometric boundary called the \emph{Bowditch boundary}, and taking the disjoint union of \(G\) with its Bowditch boundary induces a compactification called the \emph{Bowditch compactification}, which we will define shortly.
By Lemmas~\ref{lem:combing-almost-semisimple} and \ref{lem:based-markov-measure} from the preceding section, there is a based Markov measure \(\mu^0\) on the space \(\mathcal P_\infty^0\) of infinite paths in \(\mathcal G\) beginning at the initial vertex.

The goal of this section is to relate \(\mu^0\) to a natural geometric measure on the Bowditch compactification.
In Lemma~\ref{lem:compactified-bowditch-evaluation}, we show that the evaluation map from finite based paths in \(\mathcal G\) to elements of \(G\) extends to a continuous surjection from the compactified based path space \(\overline{\mathcal P}^{\,0}\) onto the Bowditch compactification.
Then, in Theorem~\ref{thm:critical-counting-convergence}, we identify the pushforward of \(\mu^0\) under this extended evaluation map with a Patterson--Sullivan measure on the Bowditch compactification.
This identification provides the measure-theoretic bridge to the ordered frequency results of the following section.

\subsection{Extended evaluation}

\begin{definition}
    Let \((G,\mathcal H)\) be a relatively hyperbolic pair with finite symmetric generating set \(S\). The \emph{Bowditch boundary} of \((G,\mathcal H)\) is given by
	\[
		\partial_B(G,\mathcal H):=\partial\widehat{\Cay}(G,S,\mathcal H)\sqcup\{v_{gH}:g\in G,\ H\in\mathcal H,\ |H|=\infty\},
	\]
    where \(\partial\widehat{\Cay}(G,S,\mathcal H)\) is the Gromov boundary of the coned--off Cayley graph of \((G,\mathcal H)\). The \emph{Bowditch compactification} of \((G,\mathcal H)\) is given by
    \[
        \overline G_B:= G \sqcup \partial_B(G,\mathcal H).
    \]
	These unions carry the subspace topology induced by Bowditch's compactification of the vertex set of the coned--off graph, not the disjoint-union topology \cite[\S\S~8--9]{Bowditch2012RelativelyHyperbolic}.
\end{definition}

The Bowditch compactification \(\overline G_B\) is a compact metrizable space \cite[\S\S~8--9]{Bowditch2012RelativelyHyperbolic} containing \(G\), with its discrete topology, as an open dense subset \cite[\S~2.3.1]{Weisman2026}.
Thus, every \(\xi\in\partial_B(G,\mathcal H)\) is the limit of a sequence of pairwise distinct elements \(g_n\in G\).
For any finite generating set \(S\), such a sequence satisfies \(\lvert g_n\rvert_S\to\infty\), since word-metric balls are finite.
Conversely, every sequence satisfying \(\lvert g_n\rvert_S\to\infty\) has a convergent subsequence by compact metrizability, and its limit belongs to \(\partial_B(G,\mathcal H)\) because every point of \(G\) is isolated in \(\overline G_B\).



The Floyd metric gives us a concrete way to attach endpoints to infinite combing paths.
It makes the tail of a geodesic ray short, so paths with a long common prefix have nearby endpoints.
Gerasimov's Floyd map theorem then carries these endpoints to the Bowditch boundary, giving the continuous extended evaluation map we will need for the measure comparison.

\begin{definition}\label{def:floyd-boundary}
Fix a function \(f(n) = F^{-n},~F >1\). For each \(v\in G\), we define a length as follows. For an edge \(e\) of \(\Cay(G,S)\), define
\[
d_S(v,e):=\min\{d_S(v,g):g\text{ is an endpoint of }e\}
\qquad\text{and}\qquad
\ell_{v}(e):=F^{-d_S(v,e)}.
\]
If \(\gamma\) is a finite edge path in \(\Cay(G,S)\), its \emph{Floyd length} is
\[
L_v(\gamma):=\sum_{e\subseteq\gamma}\ell_v(e).
\]
The induced \emph{Floyd metric} on \(G\) is then given by
\[
\rho_v(g,h):=\inf\{L_v(\gamma):\gamma\text{ is an edge path from }g\text{ to }h\}.
\]
The \emph{Floyd compactification}, denoted \( \overline G_F\) is the Cauchy completion of \(G\) with respect to \(\rho_v\), and the complement of \(G\) in \(\overline G_F\) is the \emph{Floyd boundary} associated to \(\rho_v\), denoted \(\partial_F G\). Observe \(\overline G_F: = G \sqcup \partial_FG\).
\end{definition}
\begin{proposition}\label{prop:floyd-fax}
    We list some useful facts of the Floyd metric and the Floyd boundary. We refer the reader to \cite{Floyd1980}, \cite{Gerasimov2012FloydMaps}, \cite{Karlsson2003}, and \cite{Yang2022} for more details.
    \begin{enumerate}
        \item  For any \(g\in G\) we have \(\rho_v(x,y) = \rho_{gv}(gx,gy)\).
        \item  Floyd metrics relative to different basepoints \(v\neq v' \in G\) satisfy
            \[
            F^{-d_S(v,v')} \leq \frac{\rho_v(x,y)}{\rho_{v'}(x,y)} \leq  F^{d_S(v,v')}.
            \]
        \item The action \(G\) on \(\Cay(G,S)\) extends continuously to \(\partial_F G\).
        In particular, \(G\) acts on \(\overline{G}_F\) as a convergence group action.
        \item \(\rho_v\) extends uniquely to a metric on \(\overline G_F\).
        Explicitly, if \(\xi,\eta\in\overline G_F\) are limits of Cauchy sequences \((x_n)\) and \((y_n)\) in \(G\), then
        \[
        \rho_v(\xi,\eta):=\lim_{n\to\infty}\rho_v(x_n,y_n).
        \]
        We use the same notation \(\rho_v\) for the original metric on \(G\) and its extension to \(\overline G_F\).
    \end{enumerate}
\end{proposition}

We now state Gerasimov's Floyd map theorem.

\begin{theorem}\cite[Proposition 3.4.6, Corollary~1.5]{Gerasimov2012FloydMaps}, \cite[Corollary~2.8]{GerasimovPotyagailo2013}\label{thm:floyd-map}
	Let \((G,\mathcal H)\) be a relatively hyperbolic pair and let \(S\) be a finite symmetric generating set of \(G\).
	There exists \(F_0>1\) such that, for every \(1<F\leq F_0\) and every basepoint \(v\in G\), the identity map \(\iota:(G,\rho_v)\to\overline G_B\) extends to a continuous \(G\)-equivariant surjection
	\[
		\overline \iota:\overline G_F\longrightarrow\overline G_B.
	\]
	Its restriction to the Floyd boundary is a continuous surjection
	\[
		\overline \iota|_{\partial_FG} :\partial_F G\longrightarrow\partial_B(G,\mathcal H).
	\]
\end{theorem}


To apply Theorem~\ref{thm:floyd-map}, it remains to show that the evaluation map \(\ev: \mathcal P^0 \to G\) extends continuously and surjectively to a map \(\overline{\ev}: \overline{\mathcal P}^0 \to \overline G_F\). 
For an infinite based path \(\sigma\), the natural candidate for its image is the limit in \(\overline G_F\) of the evaluated prefixes \(\ev(\sigma|_n)\).
We must therefore show that this sequence converges and that its limit varies continuously with \(\sigma\).
Both facts follow from a uniform bound on the Floyd length of the tail of a word-geodesic path: the bound guarantees convergence and shows that paths with a long common prefix have nearby images in the Floyd compactification.
The following lemma provides this required estimate.

\begin{lemma}\label{lem:floyd-tail-estimate}
	Let \(N\in\mathbb N\cup\{\infty\}\), and let \((g_k)_{0\leq k\leq N}\) be the sequence of vertices of a finite or infinite word-geodesic path in \(\Cay(G,S)\) based at \(v\in G\), parametrized so that \(d_S(v,g_k)=k\).
	For \(0\leq m<n\leq N\),
	\[
		\rho_v(g_m,g_n)\leq\sum_{k=m}^{n-1}F^{-k}\leq\frac{F^{-m}}{1-F^{-1}}.
	\]
	Consequently, every infinite word-geodesic ray based at \(v\) converges to a point of \(\partial_FG\).
\end{lemma}

\begin{proof}
	For each \(k\), the edge from \(g_k\) to \(g_{k+1}\) has Floyd length \(F^{-k}\).
	Thus the word-geodesic segment from \(g_m\) to \(g_n\) has Floyd length \(\sum_{k=m}^{n-1}F^{-k}\), and the stated estimate follows from the geometric-series formula.
	If the path is infinite, the estimate shows that \((g_n)_{n\geq 0}\) is Cauchy and hence converges to a point of \(\overline G_F\).
	Every point of \(G\) is isolated in \(\overline G_F\), while the vertices \(g_n\) are pairwise distinct, so the limit belongs to \(\partial_FG\).
\end{proof}

Now, fix \(1<F\leq F_0\) as in Theorem~\ref{thm:floyd-map}, and equip \(\Cay(G,S)\) with the corresponding Floyd metric \(\rho_{\id}\) based at \(\id\in G\).
The comparison in Proposition~\ref{prop:floyd-fax} identifies the Floyd completions for different basepoints by extending the identity on \(G\).

\begin{lemma}\label{lem:compactified-floyd-evaluation}
	The evaluation map \(\ev:\mathcal P^0\to G\) extends to a continuous surjection
	\[
		\overline{\ev}_F:\overline{\mathcal P}^{\,0}\longrightarrow\overline G_F
	\]
	defined on \(\mathcal P_\infty^0\) by
	\[
		\overline{\ev}_F(\sigma):= \lim_{n\to\infty}\ev(\sigma|_n)
	\]
\end{lemma}
\begin{proof}
    Let \((z_n)\) be a sequence in \(\overline{\mathcal P}^{\,0}\) converging to \(\sigma\in\mathcal P_\infty^0\).
    Fix \(\varepsilon>0\), and choose \(m\) sufficiently large that
    \[
    	\frac{2F^{-m}}{1-F^{-1}}<\varepsilon.
    \]
    By the definition of the prefix topology, for all sufficiently large \(n\), the paths \(z_n\) and \(\sigma\) share the prefix \(\sigma|_m\).
    If \(z_n\) is finite, Lemma~\ref{lem:floyd-tail-estimate} gives
    \[
    	\rho_{\id}\left(\overline{\ev}_F(z_n),\ev(\sigma|_m)\right)\leq\frac{F^{-m}}{1-F^{-1}}.
    \]
    If \(z_n\) is infinite, the same estimate follows by applying Lemma~\ref{lem:floyd-tail-estimate} to its finite truncations and passing to the limit in \(\overline G_F\) as in Proposition \ref{prop:floyd-fax}.
    Applying the same estimate to \(\sigma\) gives
    \[
    	\rho_{\id}\left(\ev(\sigma|_m),\overline{\ev}_F(\sigma)\right)\leq\frac{F^{-m}}{1-F^{-1}}.
    \]
    Therefore,
    \[
    	\rho_{\id}\left(\overline{\ev}_F(z_n),\overline{\ev}_F(\sigma)\right)\leq\frac{2F^{-m}}{1-F^{-1}}<\varepsilon
    \]
    for all sufficiently large \(n\).
    Thus \(\overline{\ev}_F(z_n)\to\overline{\ev}_F(\sigma)\) and \(\overline \ev_F\) is continuous.

    It remains to prove that \(\overline{\ev}_F\) is surjective.
    Let \(\xi\in\overline G_F\).
    Since \(G\) is dense in \(\overline G_F\), there is a sequence \((g_n)\) in \(G\) such that \(g_n\to\xi.\)
    For each \(n\), the bijectivity of the combing gives a based path \(y_n\in\mathcal P^0\) such that \(\ev(y_n)=g_n.\)
    Since \(\overline{\mathcal P}^{\,0}\) is compact and metrizable, there is a subsequence \((y_{n_j})\) converging to some \(z\in\overline{\mathcal P}^{\,0}\).
    By continuity of \(\overline{\ev}_F\),
    \[
    	\overline{\ev}_F(z)=\lim_{j\to\infty}\overline{\ev}_F(y_{n_j}).
    \]
    Since \(\overline{\ev}_F\) extends evaluation on finite based paths,
    \[
    	\overline{\ev}_F(y_{n_j})=\ev(y_{n_j})=g_{n_j}.
    \]
    Therefore,
    \[
    	\overline{\ev}_F(z)=\lim_{j\to\infty}g_{n_j}=\xi.
    \]
    Since \(\xi\) is arbitrary, it follows that \(\overline{\ev}_F\) is surjective.
\end{proof}

\begin{lemma}\label{lem:compactified-bowditch-evaluation}
	The map
	\[
		\overline{\ev}:=\overline \iota \circ\overline{\ev}_F:\overline{\mathcal P}^{\,0}\longrightarrow\overline G_B
	\]
	is a continuous surjection extending \(\ev:\mathcal P^0\to G\).
\end{lemma}

\begin{proof}
	Theorem~\ref{thm:floyd-map} and Lemma~\ref{lem:compactified-floyd-evaluation} show that \(\overline{\ev}\) is a composition of continuous surjections.
	Since \(\overline \iota\) restricts to the identity on \(G\), the map \(\overline{\ev}\) extends evaluation on finite paths.
\end{proof}

\begin{ainotice}
    We remark that the idea to first pass from the compactified path space to the Floyd compactification, then using the Floyd map theorem to pass to the Bowditch compactification was a suggestion of AI.
\end{ainotice}

With evaluation extended to infinite paths, we can now compare the law of a combing ray with a boundary measure obtained by weighted counting.
We follow the counting strategy of Calegari--Fujiwara \cite{CalegariFujiwara2010}, transferring each group element's weight to its unique based path in the combing.
After recalling weak convergence, we normalize these counts and show that the resulting path measures approach the based Markov law.
Continuity of extended evaluation then identifies the limit of this specific counting sequence with the distribution of endpoints.

\subsection{Weak convergence and Patterson--Sullivan measures}

We briefly recall some properties of families Borel probability measures and weak convergence which we will need. These results are standard, and they are only included to help keep this manuscript more self contained.

\begin{definition}
Let \(X\) be a metrizable space, and let \((P_n)\) and \(P\) be finite Borel measures on \(X\).
We say that \((P_n)\) \emph{converges weakly} to \(P\), written \(P_n\rightharpoonup P\), if
\[
\lim_{n\to\infty}\int_Xf\,dP_n=\int_Xf\,dP
\]
for every bounded continuous function \(f:X\to\mathbb R\).
\end{definition}

\begin{definition}
A family \(\Pi\) of Borel probability measures on a metrizable space \(X\) is \emph{tight} if, for every \(\varepsilon>0\), there exists a compact set \(K_\varepsilon\subseteq X\) such that \(P(K_\varepsilon)>1-\varepsilon\) for every \(P\in\Pi\).
\end{definition}

\begin{proposition}\label{prop:weak-sequential-compactness}
Let \(X\) be a compact metrizable space, and let \((P_n)\) be a sequence of Borel probability measures on \(X\).
Then there exist a subsequence \((P_{n_j})\) and a Borel probability measure \(P\) on \(X\) such that \(P_{n_j}\rightharpoonup P\).
\end{proposition}

\begin{proof}
The family \(\{P_n:n\geq 1\}\) is tight because one may take \(K_\varepsilon=X\) for every \(\varepsilon>0\).
Since a compact metrizable space is complete and separable, Prokhorov's theorem, for example see \cite[Theorem~5.1]{Billingsley1999}, gives a weakly convergent subsequence.
\end{proof}



Compactness gives weakly convergent subsequences, but it does not yet identify their limits.
We now specify the weighted counting sequence that we will compare with the based Markov measure.
The Poincar\'e series identifies the critical weight, and dividing each truncated count by its total mass gives a probability measure.
The growth of that normalizing mass will be used when we pass to paths.

\begin{definition}[Poincar\'e series]
The \emph{Poincar\'e series} of \((G,d_S)\) is
\[
\zeta_G(s):=\sum_{g\in G}e^{-s|g|_S}.
\]
Its \emph{critical exponent} is the supremum of the values of \(s\) for which this series diverges.
\end{definition}

Because every vertex of \(\mathcal G\) is accessible from \(v_0\), the exponential growth rate of \(|\mathcal P_n^0|=\mathbf e_0^TM^n\mathbf1\) is \(\lambda\). The bijection \(\mathcal P_n^0\leftrightarrow G_n\) and Lemma~\ref{lem:pure-sphere-growth} therefore show that the critical exponent is \(\log\lambda\) and that the Poincar\'e series diverges at this value. Set
\begin{equation}\label{eq:critical-normalizing-constant}
Z_n:=\sum_{|g|_S\leq n}\lambda^{-|g|_S}
=\sum_{k=0}^n\lambda^{-k}|G_k|
=\sum_{k=0}^n\lambda^{-k}|\mathcal P_k^0|.
\end{equation}
The associated probability measure on \(G\) is
\begin{equation}\label{eq:critical-counting-probabilities}
\nu_n:=\frac{1}{Z_n}\sum_{|g|_S\leq n}\lambda^{-|g|_S}\delta_g.
\end{equation}
We also regard \(\nu_n\) as a measure on the Bowditch compactification \(\overline G_B=G\sqcup\partial_B(G,\mathcal H)\) that assigns no mass to \(\partial_B(G,\mathcal H)\).

\begin{definition}\label{def:critical-counting-ps-measure}
A \emph{Patterson--Sullivan measure} for \((G,d_S)\) is any weak limit of \((\nu_n)_{n\geq0}\) on \(\overline G_B\). We call each \(\nu_n\) an \emph{approximate} of a Patterson--Sullivan measure.
\end{definition}

\begin{lemma}\label{lem:critical-normalization}
    The limit as \(n\to\infty\) of \(Z_n/n\) exists and satisfies
    \[
        \lim_{n\to\infty}\frac{Z_n}{n}=r_{v_0}>0,
    \]
    where \(r_{v_0}=\mathbf e_0^Tr\).
\end{lemma}

\begin{proof}
Equation~\eqref{eq:based-path-count} and Lemma~\ref{lem:cesaro-pf-projections} give
\[
\frac{Z_n}{n}
=\mathbf e_0^T\left(\frac{1}{n}\sum_{k=0}^n\lambda^{-k}M^k\mathbf1\right)
\longrightarrow\mathbf e_0^Tr=r_{v_0}.
\]
The lower bound in Lemma~\ref{lem:pure-sphere-growth} makes \(\lambda^{-k}|G_k|\) uniformly positive for \(k\geq1\). Thus \(r_{v_0}>0\).
\end{proof}

\subsection{Lifting Patterson--Sullivan approximates to paths}


To lift a counting measure, we lift each group element's mass to its unique based path in the combing.
This changes where the mass is recorded, not the weighting.
The normalization we computed above lets us compare the mass of paths sharing a fixed prefix with the probability of that prefix under the based Markov law.
We now compare the lifted measures with \(\mu^0\) on \(\overline{\mathcal P}^{\,0}=\mathcal P^0\sqcup\mathcal P_\infty^0\), extending \(\mu^0\) by zero on finite paths.

\begin{lemma}\label{lem:lifted_measures}
	Each measure \(\nu_n\) on \(\overline G_B\) has a unique lift to a measure \(\eta_n\) on \(\overline{\mathcal P}^0\) supported on \(\mathcal P^0\).
	Moreover,
    \[
        \eta_n \rightharpoonup \mu^0,
    \]
    where \(\mu^0\) is the based Markov measure from Lemma \ref{lem:based-markov-measure} supported on \(\mathcal P^0_\infty\).
\end{lemma}

\begin{proof}
    For each \(g\in G\), let \(y_g\in\mathcal P^0\) be the unique based path satisfying \(\ev(y_g)=g\).
    Let
    \[
        \iota:=\left(\overline \ev|_{\mathcal P^0}\right)^{-1}:G\longrightarrow\mathcal P^0.
    \]
    Observe, this is a bijection as the restriction of \(\overline \ev\) to \(\mathcal P^0\) is simply the map \(\ev\), which is a bijection.  Since \(\nu_n\) is supported on \(G\), define its lift to \(\overline{\mathcal P}^{\,0}\) by
    \[
        \eta_n:=\iota_*\nu_n.
    \]
    Thus,
    \[
        \eta_n = \frac{1}{Z_n} \sum_{\lvert g\rvert_S\leq n} \lambda^{-\lvert g\rvert_S}\delta_{y_g}, \qquad y_g=\iota(g).
    \]
    The measure \(\eta_n\) is supported on \(\mathcal P^0\), and the bijectivity of \(\ev:\mathcal P^0\to G\) gives
    \[
        (\overline{\ev})_*\eta_n=\nu_n.
    \]
    Moreover, this is the unique lift of \(\nu_n\) supported on \(\mathcal P^0\), since the mass of each finite path \(y_g\) must equal \(\nu_n(\{g\})\). Let \(y\in\mathcal P_k^0\) terminate at \(v_i\), and let
    \[
    \mathcal C(y):=\{z\in\overline{\mathcal P}^{\,0}:z\text{ extends }y\}.
    \]
	Since \(\mathcal C(y)\cap\mathcal P_\infty^0=[y]\), the zero extension gives \(\mu^0(\mathcal C(y))=\mu^0([y])\).
    For \(n\geq k\), the finite paths in \(\mathcal C(y)\) having length at most \(n\) are obtained by extending \(y\) by a path of length \(q\), where \(0\leq q\leq n-k\). Proposition~\ref{prop:number_of_paths} therefore gives
    \[
        \eta_n(\mathcal C(y)) = \frac{\lambda^{-k}}{Z_n}\sum_{q=0}^{n-k}\lambda^{-q}\mathbf e_i^TM^q\mathbf 1.
    \]
    Since \(k\) is fixed, Lemmas~\ref{lem:cesaro-pf-projections} and \ref{lem:critical-normalization} give
	\[
		\eta_n(\mathcal C(y))=\lambda^{-k}\frac{\frac{1}{n}\sum_{q=0}^{n-k}\lambda^{-q}\mathbf e_i^TM^q\mathbf 1}{Z_n/n}\longrightarrow\frac{\lambda^{-k}r_{v_i}}{r_{v_0}}=\mu^0([y])=\mu^0(\mathcal C(y)).
	\]
    If \(z\in\mathcal P^0\) is a fixed finite path, then for all \(n\geq |z|\),
    \[
        \eta_n(\{z\})=\frac{\lambda^{-|z|}}{Z_n}\longrightarrow 0=\mu^0(\{z\}),
    \]
    because \(Z_n\to\infty\).
    
	Since \(\overline{\mathcal P}^{\,0}\) is compact and metrizable, Proposition~\ref{prop:weak-sequential-compactness} implies that every subsequence of \((\eta_n)\) has a weakly convergent further subsequence.
	Let \((\eta_{n_j})\) be any weakly convergent subsequence, and write \(\eta_{n_j}\rightharpoonup\eta\).
	The finite-path singletons together with the sets \(\mathcal C(y)\) form a clopen \(\pi\)-system that generates the Borel \(\sigma\)-algebra of \(\overline{\mathcal P}^{\,0}\).
	Since every set \(A\) in this family is clopen, weak convergence (and observing the indicator function for a clopen set is continuous) gives \(\eta_{n_j}(A)\to\eta(A)\).
	The preceding calculations also give \(\eta_{n_j}(A)\to\mu^0(A)\), and hence \(\eta(A)=\mu^0(A)\).
	The uniqueness theorem for measures therefore gives \(\eta=\mu^0\).
	Thus every subsequence of \((\eta_n)\) has a further subsequence converging weakly to \(\mu^0\), which implies that
	\[ \eta_n\rightharpoonup\mu^0.\qedhere \] 
\end{proof}

\begin{theorem}\label{thm:critical-counting-convergence}
	\label{thm:intro-markov-ps}
	Let \(\nu:=(\overline{\ev})_*\mu^0\).
	Then the measures \(\nu_n\) converge weakly to \(\nu\) on \(\overline G_B\).
	Moreover, \(\nu\) is a Patterson--Sullivan measure and is supported on \(\partial_B(G,\mathcal H) \subseteq \overline G_B\).
\end{theorem}

\begin{proof}
	Let \(f:\overline G_B\to\mathbb R\) be bounded and continuous.
	Since \(\overline{\ev}\) is continuous, the function \(f\circ\overline{\ev}\) is bounded and continuous on \(\overline{\mathcal P}^{\,0}\).
	Using \((\overline{\ev})_*\eta_n=\nu_n\) and \(\eta_n\rightharpoonup\mu^0\) from Lemma \ref{lem:lifted_measures}, we obtain
	\[ 
        \int_{\overline G_B}f\,d\nu_n=\int_{\overline{\mathcal P}^{\,0}}f\circ\overline{\ev}\,d\eta_n\longrightarrow\int_{\overline{\mathcal P}^{\,0}}f\circ\overline{\ev}\,d\mu^0=\int_{\overline G_B}f\,d\nu.
    \]
	Therefore \(\nu_n\rightharpoonup\nu\).
	Since \(\mu^0\) is supported on \(\mathcal P_\infty^0\) and \(\overline{\ev}(\mathcal P_\infty^0)\subseteq\partial_B(G,\mathcal H)\), the measure \(\nu\) is supported on the Bowditch boundary.
	The conclusion now follows from Definition~\ref{def:critical-counting-ps-measure}.
\end{proof}

The pushforward identity also transfers any Borel full-measure set of paths with some property to a Borel full-measure set of boundary points.
We have now identified the limit of this specific weighted counting sequence as the endpoint distribution of a random based combing path.
The next section studies ordered word frequencies for that path law.
Once those frequencies exist almost surely, the measure identity tells us that almost every boundary point, for this counting-limit measure, has at least one combing ray along which they all exist.

\section{Frequency}\label{sec:frequency}


The measure identity from the preceding section lets us interpret almost-sure statements about combing paths as statements about rays ending at almost every boundary point.
We now study word frequencies under the based Markov measure.
We first recall where a typical path eventually stays, then follow the period-restricted counting argument of Basu and Mj before treating arbitrary word lengths.

\subsection{Graph structure components}
\begin{definition}\label{def:automaton-components}
     Let \(\mathcal{G}\) be a graph structure. Two vertices of \(\mathcal G\) lie in the same \emph{component} if each is accessible from the other.
A component is \emph{recurrent} if it contains a directed cycle.
\end{definition}

The directed graph obtained by collapsing each component of \(\mathcal G\) to a vertex has no directed cycles.
Let \(M\) be the adjacency matrix of \(\mathcal G\), where \(M_{ij} =1\) if there is a directed edge from \(v_i\) to \(v_j\) and 0 otherwise. (In general, \(M_{ij}\) is the number of edges from \(v_i\) to \(v_j\), but we assume \(\mathcal G\) has no parallel edges).
For a component \(C\), let \(M_C\) be its adjacency matrix and let \(\lambda_C\) be its spectral radius.
By Perron-Frobenius, \(\lambda\) is in fact the largest eigenvalue.
Set \(\lambda=\max_C\lambda_C\), and call \(C\) \emph{maximal} if \(\lambda_C=\lambda\).

\begin{definition}\label{def:component-period}
The \emph{period} \(p(C)\) of a recurrent component \(C\) is the greatest common divisor of the lengths of all directed cycles contained in \(C\).
\end{definition}


Let \(\mathcal P (C) \subseteq \mathcal P\) and \(\mathcal P_\infty(C) \subseteq \mathcal P_\infty\) denote the collections of finite and infinite paths contained in a maximal component \(C\).

\begin{proposition}\label{cor:ev-in-maximal-comp}\cite[Corollary 3.15]{BasuMj2022}
 For \(\mu^0\)-a.e. \(\sigma\in \mathcal P^0_\infty\), there is a unique maximal component \(C\) so that \(\mathsf S^i(\sigma) \in \mathcal{P}_\infty(C)\) for all \(i\) sufficiently large.
\end{proposition}

\subsection{Markov chain trajectories}

Since a typical path eventually stays in a maximal component, we next recall a counting result for an irreducible finite Markov chain.
We count consecutive, nonoverlapping tuples, keeping the period restriction used in the mixing argument of Basu and Mj.

Let \(P\) be the transition matrix of an irreducible Markov chain on a finite state space \(\Sigma\). We do not assume that the chain is aperiodic or reversible, and we denote its period by \(d\). Fix \(k\in d\mathbb{N}\) and a word \(\mathbf{x}=(x_1,\ldots,x_k)\in\Sigma^k\). For each \(x\in\Sigma\), let \((X_m^x)_{m\geq 0}\) be a realization of the chain with \(X_0^x=x\). For \(n\geq k\), define
\[
N_n(\mathbf{x},x):=\#\left\{0\leq i\leq \left\lfloor\frac{n}{k}\right\rfloor-1:(X_{ik+1}^x,\ldots,X_{(i+1)k}^x)=\mathbf{x}\right\}.
\]
Thus \(N_n(\mathbf{x},x)\) counts the occurrences of \(\mathbf{x}\) in the consecutive, nonoverlapping tuples of length \(k\) observed up to time \(n\).

\begin{proposition}[\cite{BasuMj2022}\label{prop:counts-of-word}, Proposition 3.16]
    For each $x\in \Sigma$, the following holds almost surely. For all $k\in d\mathbb{N}$, $\mathbf{x}\in \Sigma^k$, and $\varepsilon>0$, there exist $f(\mathbf{x},x) \geq 0$ (non-random), and $N_0$ (random but finite, depending on $k,\mathbf{x},\varepsilon$) such that for all $n\geq N_0$
    \[
    \frac{N_n(\mathbf{x},x)}{n} \in (f(\mathbf{x},x) - \varepsilon, f(\mathbf{x},x) + \varepsilon,)
    \]
    i.e.
    \[
    \frac{N_n(\mathbf{x},x)}{n} \to f(\mathbf{x},x)
    \]
    almost surely.
\end{proposition}

Basu and Mj give a proof in \cite[Appendix A]{BasuMj2022}.
They use the fact that when the tuple length is divisible by the period of the original chain, the corresponding tuple process is aperiodic, and hence mixing, on each of its recurrent components.

\subsection{Ordered frequency}


We now apply the tuple counts to the words spelled by a combing path.
For a fixed word, we inspect consecutive subpaths of the same length and count those whose labels spell that word.
The definition makes sense for every positive length; the first existence proof below uses the period restriction.

We now move to the notion of ordered frequency. We briefly recall our setting.  Let \((G,\mathcal H)\) be a non-elementary relatively hyperbolic pair, with admissible finite symmetric generating set \(S\), and let \(\mathcal G\) be a geodesic combing of \((G,S)\) with initial vertex \(v_0\).



Let \(\mathcal{C}_{\max}\) denote the collection of maximal components of \(\mathcal G\). For each \(C\in\mathcal{C}_{\max}\), let \(d_C\) denote the period of the irreducible Markov chain supported on \(C\), and set
\[
d:=\operatorname{lcm}\{d_C:C\in\mathcal{C}_{\max}\}.
\]
By Proposition~\ref{cor:ev-in-maximal-comp}, for \(\mu^0\)-almost every \(\sigma\in\mathcal{P}_\infty^0\), there exists a unique maximal component \(C\) such that \(\sigma\) eventually remains in \(\mathcal{P}_\infty(C)\). Let \(\mathcal{P}_\infty^{\max}\subseteq\mathcal{P}_\infty^0\) denote the collection of all such infinite paths.

Fix \(n\in d\mathbb{N}\), let \(w=s_1s_2\cdots s_n\) be a geodesic word of length \(n\), and let \(\sigma\in\mathcal{P}_\infty^{\max}\). For each \(i\geq 0\), let \(y_i\) denote the prefix of \(\sigma\) of length \(in\). In particular, \(y_0\) is the empty path and \(\ev(y_0)=\id\). Since \(\sigma\) is geodesic,
\[
d_S\bigl(\ev(y_i),\ev(y_{i+1})\bigr)=n
\]
for every \(i\geq 0\). We define the ordered frequency of \(w\) along \(\sigma\) by counting those segments from \(y_i\) to \(y_{i+1}\) whose ordered edge labels form the word \(w\). We now give a precise definition.

\begin{definition}[Ordered frequency]
Let \(\sigma\in\mathcal P_\infty^0\), let \(w=s_1\cdots s_k\) be a geodesic word of length \(k\), and, for each \(j\geq 0\), let \(y_j\) be the prefix of \(\sigma\) of length \(jk\). We say that \(w\) has an \emph{ordered occurrence at \(j\)} if the ordered labels on the subpath from \(y_j\) to \(y_{j+1}\) form the \(k\)-tuple \((s_1,\ldots,s_k)\).
Equivalently, if \(z_j\) is the unique length-\(k\) subpath satisfying \(y_{j+1}=y_jz_j\), then the ordered labels of \(z_j\) form \(w\).
For \(m\geq 1\), define
\[
n_w(\sigma;m)
:=
\#\left\{0\leq j<m:\text{\(w\) has an ordered occurrence at \(j\)}\right\}.
\]
The \emph{ordered frequency} of \(w\) along \(\sigma\) is
\[
f_w(\sigma)
:=
\lim_{m\to\infty}\frac{n_w(\sigma;m)}{m},
\]
provided the limit exists.
\end{definition}

\begin{lemma}[Ordered frequencies exist]\label{lem:ordered-frequencies-exist}
For \(d\) and \(\mathcal P_\infty^{\max}\) as above, for \(\mu^0\)-almost every \(\sigma\in\mathcal P_\infty^{\max}\), the ordered frequency \(f_w(\sigma)\) exists for every geodesic word \(w\) whose length is a positive multiple of \(d\).
\end{lemma}
\begin{proof}
By Proposition~\ref{cor:ev-in-maximal-comp}, for \(\mu^0\)-almost every \(\sigma\in\mathcal P_\infty^0\), there is a unique maximal component \(C\) and an integer \(J\geq 0\) such that every vertex of \(\sigma\) after time \(J\) lies in \(C\).
Fix such a path \(\sigma\) , and let \(\mathsout{w}\) be a geodesic word of length \(k\in d\mathbb N\).

Choose \(j_0\) so that \({j_0k\geq J}\); then the subpath of \(\sigma\) beginning at \(y_{j_0}\) lies entirely in \(C\).
Write \(\sigma=(v_0,v_1,v_2,\ldots)\), so that the prefix \(y_{j_0}\) terminates at \(v_{j_0k}\). Conditional on the cylinder \([y_{j_0}]\), the sequence of automaton vertices in the tail of \(\sigma\) has the law of the Markov chain induced by \(N\) on \(C\), started at \(v_{j_0k}\).
This chain is irreducible and has period \(d_C\), and \(k\) is a multiple of \(d_C\) because \(d_C\) divides \(d\).






For each \(j\geq j_0\), let \(z_j\) be the unique length-\(k\) subpath satisfying \(y_{j+1}=y_jz_j\). Since the tail of \(\sigma\) after \(y_{j_0}\) lies in \(C\), each \(z_j\) is contained in \(C\). The sequence
\[
(z_{j_0},z_{j_0+1},z_{j_0+2},\ldots)
\]
	is the trajectory of the finite Markov chain of consecutive length-\(k\) tuples in \(\mathsout{C}\).
    Since \(k\) is divisible by \(d_C\), Proposition~\mbox{\ref{prop:counts-of-word}} therefore implies that
\[
\lim_{m\to\infty}
\frac{n_w(\sigma;j_0+m)-n_w(\sigma;j_0)}{m}
\]
exists almost surely. For every \(m\geq 1\),
\[
\frac{n_w(\sigma;j_0+m)}{j_0+m}
=
\frac{m}{j_0+m}
\frac{n_w(\sigma;j_0+m)-n_w(\sigma;j_0)}{m}
+
\frac{n_w(\sigma;j_0)}{j_0+m}.
\]
The first factor tends to one, the middle quotient has an almost-sure limit, and the final term tends to zero as \(m\to\infty\). Therefore,
\[
\lim_{m\to\infty}\frac{n_w(\sigma;j_0+m)}{j_0+m}
=
f_w(\sigma)
\]
exists almost surely.
\end{proof}

Let \(\mathcal P'_\infty\subset\mathcal P^0_\infty\) be the \(\mu^0\)-full measure subset satisfying Lemma~\ref{lem:ordered-frequencies-exist}, and let \(\partial_B(G,\mathcal H)':=\overline{\ev}(\mathcal P'_\infty)\).
Since \(\nu=(\overline{\ev})_*\mu^0\), this set has full \(\nu\)-measure.
We therefore obtain the following corollary.

\begin{corollary}\label{cor:boundary_frequencies}
	For \(\nu\)-almost every \(\xi\in\partial_B(G,\mathcal H)\), there exists a combing ray \(\sigma\in\mathcal P'_\infty\) ending at \(\xi\). 
\end{corollary}

\subsection{Ordered frequency without mixing}


The preceding corollary gives the boundary conclusion for the word lengths covered by the mixing argument.
We have kept that argument to make the comparison with Basu and Mj explicit.
We now remove the length restriction on the path space; the same measure identity will then give the boundary conclusion for all word lengths.

Basu and Mj observe in \cite[Remark 3.19]{BasuMj2022} that ordered frequencies can be defined without the divisibility assumption.
They impose this assumption to obtain mixing for the associated tuple process.
For the existence of ordered frequencies, however, mixing is unnecessary: the finite state Markov chain ergodic theorem applied to a closed irreducible component eventually entered by the finite tuple chain suffices.
Thus, we can show ordered frequencies exist almost surely for geodesic words of arbitrary length.

\begin{theorem}[Ordered frequencies for arbitrary word lengths]\label{thm:arbitrary-length-ordered-frequencies}
	\label{thm:intro-ordered-frequency}
	For \(\mu^0\)-almost every \(\sigma\in\mathcal P_\infty^0\), the ordered frequency \(f_w(\sigma)\) exists for every nonempty geodesic word \(w\).
\end{theorem}

\begin{proof}
    By Lemma~\ref{lem:based-markov-measure}, a path chosen according to \(\mu^0\) starts at \(v_0\) and moves from \(v_i\) to \(v_j\) with probability \(N_{ij}\), independently of its earlier vertices.
	Fix \(k\geq1\), let \(y_j\) be the prefix of length \(jk\), and let \(z_j\) be the length-\(k\) subpath satisfying \(y_{j+1}=y_jz_j\).
    This gives rise to a finite state tuple process.
    The possible states of the tuple process are the length-\(k\) paths that occur with positive probability along a path chosen according to \(\mu^0\).
	These consecutive tuples form a Markov chain because the conditional distribution of the next tuple depends only on the last vertex of the current tuple, and hence only on the current tuple.
	By the finiteness of the chain, almost surely, this tuple chain eventually remains in a closed irreducible component.
	The finite state Markov chain ergodic theorem~\cite[Theorem~10.2]{Weber2011MarkovChains} gives convergence of every empirical tuple proportion to its stationary probability in its respective closed irreducible component.
	This applies from every possible entrance state and time, and the random but finitely many preceding tuples, which may lie outside the eventual closed irreducible component, do not affect the limiting proportions.
	For a word \(w\) of length \(k\), the quotient \(n_w(\sigma;m)/m\) is the sum of the empirical proportions of the finitely many tuples whose ordered edge labels form \(w\).
	Consequently, the limit
	\[
		\lim_{m\to\infty}\frac{n_w(\sigma;m)}{m}=f_w(\sigma)
	\]
	exists almost surely.
	Taking the countable intersection of the full-measure sets for all nonempty geodesic words proves that the ordered frequency of every nonempty geodesic word must exist almost surely.
\end{proof}

    Let \(\mathcal P''_\infty\subset\mathcal P^0_\infty\) be the \(\mu^0\)-full measure subset satisfying Theorem~\ref{thm:arbitrary-length-ordered-frequencies}, and let \(\partial_B(G,\mathcal H)'':=\overline{\ev}(\mathcal P''_\infty)\).
    Since \(\nu=(\overline{\ev})_*\mu^0\), this set has full \(\nu\)-measure.
    We therefore obtain the following corollary.

\begin{corollary}
	\label{cor:arbitrary-length-boundary-frequencies}
	For \(\nu\)-almost every \(\xi\in\partial_B(G,\mathcal H)\), there exists a combing ray \(\sigma\in\mathcal P^0_\infty\) ending at \(\xi\) along which the ordered frequency exists for every nonempty geodesic word.
\end{corollary}

\appendix
\refstepcounter{section}

\section*{Appendix}\label{app:worked-examples}

The worked example in this appendix puts the graph structure, the measure, and the ordered frequencies side by side.
For \(\mathbb Z*\mathbb Z^2\), we compute the Perron--Frobenius data and compare exact limiting frequencies with a recorded finite sample.

\begin{ainotice}
	The results of this section were achieved almost entirely through the use of agentic programming.
	We view this as perhaps the best use of AI in this manuscript.
	The currently available tools are particularly useful for implementing explicit constructions, computing examples, and producing mathematical figures.
	Here, they helped us calculate the Perron--Frobenius data, simulate geodesic rays, and compare empirical ordered frequencies with their predicted values.
	These computations illustrate the mathematical arguments rather than replace them, and responsibility for their interpretation and correctness remains with the author.
    The code for this is freely available at \url{https://github.com/DothWalrus/automaton_of_F_2/blob/main/geodesic_combing_zz2.ipynb}.
\end{ainotice}

\subsection{A geodesic combing of \(\mathbb Z*\mathbb Z^2\)}\label{subsec:example-zz2}

\paragraph{The chosen normal form language and its geometry.}
Consider the group and symmetric generating set
\[
	G=\langle a,b,c\mid bc=cb\rangle\cong\mathbb Z*\mathbb Z^2,\qquad S=\{a,a^{-1},b,b^{-1},c,c^{-1}\}.
\]
Set \(H=\langle b,c\rangle\cong\mathbb Z^2\) and take the peripheral family \(\mathcal H=\{H\}\).
The group \(G\) is hyperbolic relative to \(H\), so this gives a relatively hyperbolic pair \((G,\mathcal H)\) as in the introduction.
The reduced free-product normal form expresses each nonidentity element uniquely as an alternating sequence of nonidentity syllables from \(\langle a\rangle\) and \(\langle b,c\rangle\).
Choose the representative \(a^q\) for a syllable in the first factor, where \(q\ne0\), and \(b^m c^n\) for a syllable in the second, where \((m,n)\ne(0,0)\).
Here a negative exponent denotes the corresponding repeated inverse letter, and a zero exponent contributes no letter.
Our language consists of these alternating concatenations together with the empty word.
Uniqueness of the factor normal form and of the chosen representative in each factor gives bijective evaluation onto \(G\).
The chosen syllables have lengths \(|q|\) and \(|m|+|n|\), respectively.
Reducing any representative to its factor normal form cannot increase its length, while our representative achieves the sum of these syllable lengths, so every chosen word is geodesic.
Every prefix is again a chosen word.

\paragraph{The geodesic combing.}
Use the ordered state set
\[
	(v_0,v_a,v_{a^{-1}},v_b,v_{b^{-1}},v_c,v_{c^{-1}}).
\]
\(v_0\) is the initial state, and \(v_s\) records the last letter read.
An edge ending at \(v_s\) has label \(s\), so the destination specifies the label of every edge in the following adjacency matrix:
\[
	M=\begin{pmatrix}0&1&1&1&1&1&1\\0&1&0&1&1&1&1\\0&0&1&1&1&1&1\\0&1&1&1&0&1&1\\0&1&1&0&1&1&1\\0&1&1&0&0&1&0\\0&1&1&0&0&0&1\end{pmatrix}.
\]
Thus an \(a\)-run keeps its sign until a letter from the other factor is read, and a \(b\)-run may be followed by a \(c\)-run of either sign.
A \(c\)-run can continue with the same sign or enter the \(a\)-factor, but it cannot return directly to a \(b\)-letter.
These rules accept exactly the normal forms described above, proving regularity as well as prefix closure.

Every state is reachable from \(v_0\), and any finite path in the geodesic combing can therefore be preceded by a based path; its label is a suffix of a geodesic word and is itself geodesic.
The matrix and the destination-label convention completely specify the labeled graph in Figure~\ref{fig:example-zz2-automaton}.

\begin{figure}[htbp]
	\centering
	\includegraphics[width=\linewidth]{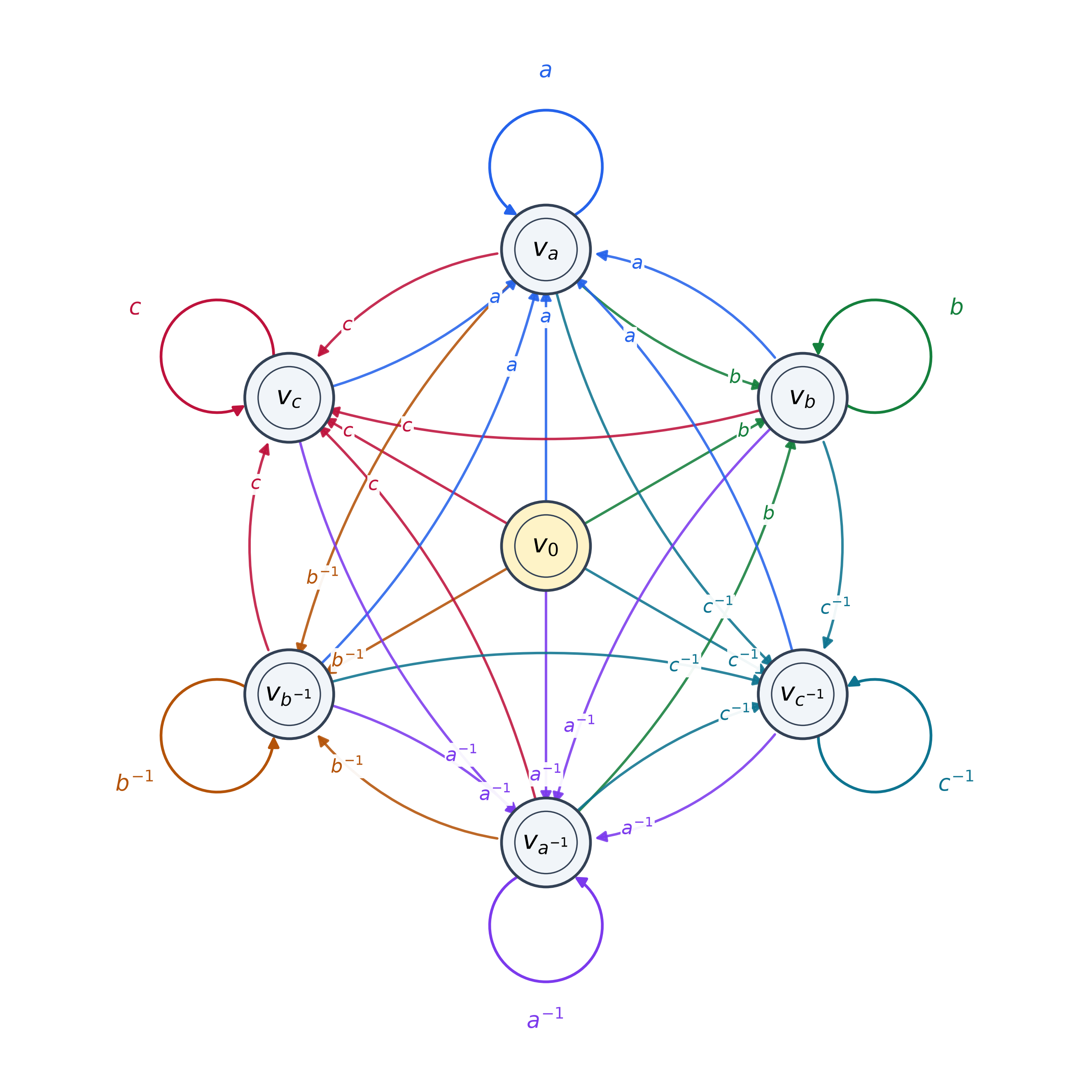}
	\caption{The seven-state geodesic combing for the chosen normal forms in \(\mathbb Z*\mathbb Z^2\).
	The six noninitial states form a single recurrent component, with a loop at each state.}
	\label{fig:example-zz2-automaton}
\end{figure}

Evaluating the based paths of this geodesic combing produces the geodesic combing rays in the Cayley graph.
Figure~\ref{fig:example-zz2-cayley} shows the resulting combing tree in a finite Cayley ball.
The order of letters matters: \(bc\) and \(cb\) represent the same element and are both geodesic, but only \(bc\) belongs to the chosen language.
In particular, ordered word frequencies distinguish spellings, not merely their evaluations in \(G\).

\begin{figure}[htbp]
	\centering
	\includegraphics[width=\linewidth]{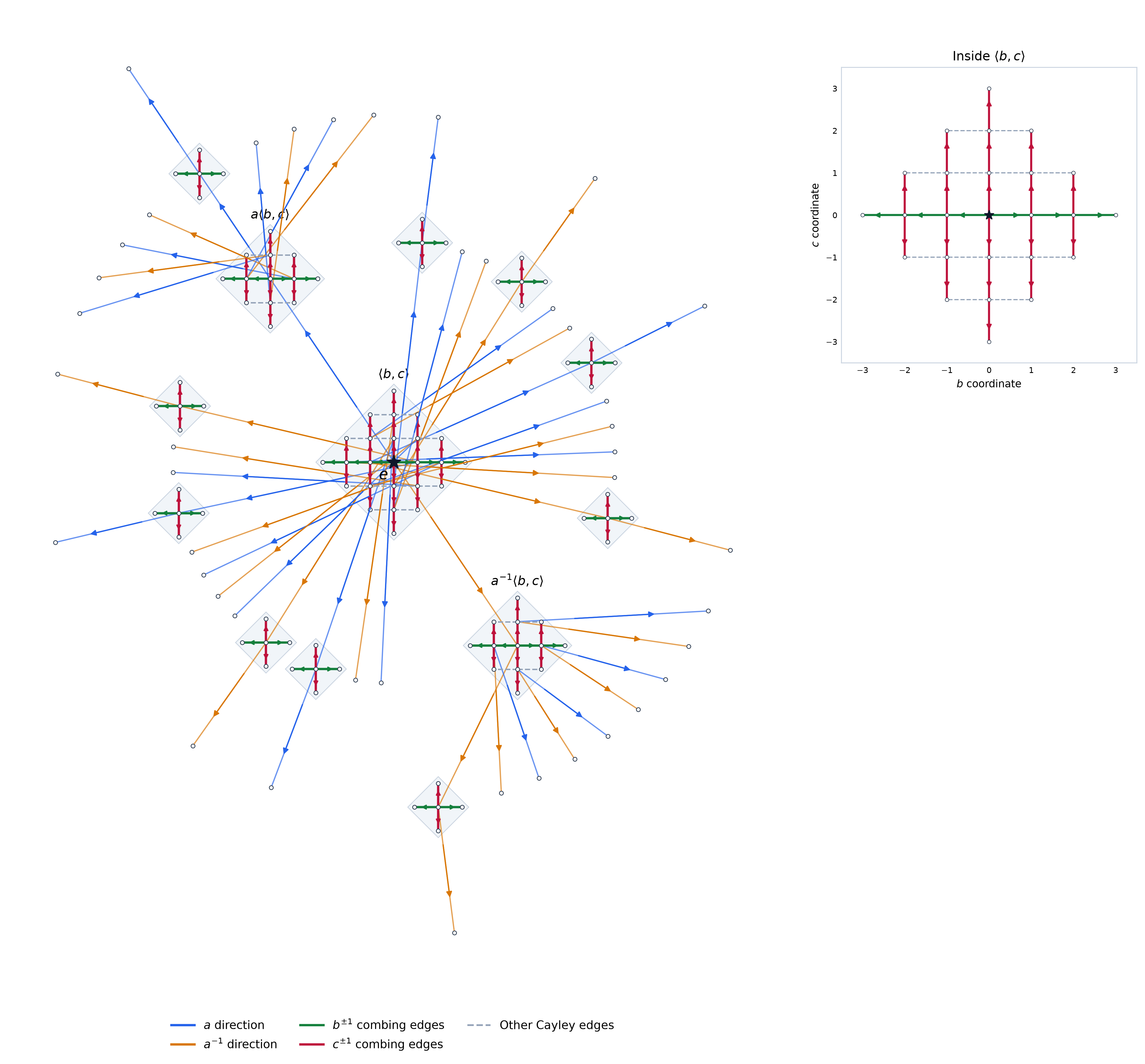}
	\caption{The radius-three Cayley ball for \(\mathbb Z*\mathbb Z^2\), centered at the identity.
	Solid colored edges carry arrows away from the identity along chosen geodesics; dashed edges are the remaining Cayley edges.
	Within each visible grid coset of \(\langle b,c\rangle\), the combing reads the \(b\)-letters before the \(c\)-letters, as illustrated by the close-up.
    }
	\label{fig:example-zz2-cayley}
\end{figure}

\paragraph{Exact Perron--Frobenius projections.}
The components are the transient singleton \(\{v_0\}\) and the terminal component \(C=\{v_a,v_{a^{-1}},v_b,v_{b^{-1}},v_c,v_{c^{-1}}\}\).
The latter is strongly connected and has loops, so its period is one.
The component quotient is the single arrow \(\{v_0\}\longrightarrow C\), and every infinite based path enters \(C\) after its first edge.
Direct calculation gives
\[
	\det(zI-M)=z(z-1)^3(z+1)(z^2-4z-1),\qquad \lambda=2+\sqrt5,\qquad \varphi=\frac{1+\sqrt5}{2}.
\]
In particular, \(\lambda\) is simple and strictly dominant, and \(C\) is the unique maximal component.
For a transparent computation of the projection, take auxiliary right and left eigenvectors
\[
	u=(2,\varphi,\varphi,\varphi,\varphi,1,1)^T,\qquad v=(0,\varphi,\varphi,1,1,\varphi,\varphi)^T.
\]
They satisfy \(Mu=\lambda u\), \(M^Tv=\lambda v\), and \(v^Tu=5+3\sqrt5\).
Consequently, \(\lambda^{-n}M^n\) converges to \(P=uv^T/(v^Tu)\), and the Ces\`aro projections of Lemma~\ref{lem:cesaro-pf-projections} are
\[
	r=P\mathbf1=\frac{5+\sqrt5}{10}(2,\varphi,\varphi,\varphi,\varphi,1,1)^T.
\]
\[
	\ell=P^T\mathbf e_0=\frac{3\sqrt5-5}{10}(0,\varphi,\varphi,1,1,\varphi,\varphi)^T.
\]
These are the projection-defined vectors, not eigenvectors rescaled to have coordinate sum one.
For example, their expanded coordinates are
\begin{align*}
	r_0&=1+\frac{\sqrt5}{5},\qquad r_c=r_{c^{-1}}=\frac{5+\sqrt5}{10},\\
	r_a=r_{a^{-1}}=r_b=r_{b^{-1}}&=\frac{5+3\sqrt5}{10}.
\end{align*}
\begin{align*}
	\ell_0&=0,\qquad \ell_b=\ell_{b^{-1}}=\frac{3\sqrt5-5}{10},\\
	\ell_a=\ell_{a^{-1}}=\ell_c=\ell_{c^{-1}}&=\frac{5-\sqrt5}{10}.
\end{align*}
Here and below, a subscript \(s\) on a vector denotes the coordinate at \(v_s\).
These values agree with the unnormalized vectors in the notebook, and satisfy \(\sum_i r_i\ell_i=r_0\), as also follows from \(P^2=P\).
Moreover, \(|\mathcal P_n^0|=\mathbf e_0^TM^n\mathbf1\sim r_0\lambda^n\), which verifies almost semisimplicity directly in this example.

\paragraph{Based and stationary probabilities.}
All coordinates of \(r\) are positive, so no artificial absorbing row is needed in \(N\).
Equation~\eqref{eq:state-transition-probability} gives the full transition matrix in the row and column order \((v_0,v_a,v_{a^{-1}},v_b,v_{b^{-1}},v_c,v_{c^{-1}})\):
\[
	N=\frac{1}{\lambda}\begin{pmatrix}0&\varphi/2&\varphi/2&\varphi/2&\varphi/2&1/2&1/2\\0&1&0&1&1&\varphi^{-1}&\varphi^{-1}\\0&0&1&1&1&\varphi^{-1}&\varphi^{-1}\\0&1&1&1&0&\varphi^{-1}&\varphi^{-1}\\0&1&1&0&1&\varphi^{-1}&\varphi^{-1}\\0&\varphi&\varphi&0&0&1&0\\0&\varphi&\varphi&0&0&0&1\end{pmatrix}.
\]
Every row sums to one because \(\lambda=2\varphi+1=3+2/\varphi\).
The stationary vertex distribution is
\[
	(\mu_0,\mu_a,\mu_{a^{-1}},\mu_b,\mu_{b^{-1}},\mu_c,\mu_{c^{-1}})=\left(0,\frac{\sqrt5}{10},\frac{\sqrt5}{10},\frac{5-\sqrt5}{20},\frac{5-\sqrt5}{20},\frac{5-\sqrt5}{20},\frac{5-\sqrt5}{20}\right).
\]
It sums to one and satisfies \(N^T(\mu_i)_i=(\mu_i)_i\), either by substitution or by using the left eigenvector equation for \(\ell\).
The stationary path law \(\mu\) starts in \(C\) with this distribution, whereas the based law \(\mu^0\) starts at \(v_0\) and uses the same transition probabilities.
For a based path \(y\) of positive length \(k\) ending at \(v_s\), telescoping the transition probabilities gives

\[
	\mu^0([y])=\frac{r_s}{\lambda^k r_0}=\begin{cases}\varphi/(2\lambda^k),&s\in\{a,a^{-1},b,b^{-1}\},\\1/(2\lambda^k),&s\in\{c,c^{-1}\}.\end{cases}
\]

In particular, under \(\mu^0\) the probability that the first letter is \(a\) is \(\varphi/(2\lambda)=(3-\sqrt5)/4\), whereas under \(\mu\) it is \(\mu_a=\sqrt5/10\).
The latter equality uses stationarity and the fact that reading \(a\) is exactly arrival at \(v_a\).
This difference is why a based cylinder probability is not, in general, the limiting ordered frequency of its label word.

\paragraph{Ordered word frequencies.}
For a nonempty word \(w=s_1\cdots s_k\), let \(p_\mu(w)\) denote the probability under the stationary path law that the next \(k\) labels spell \(w\).
For this example, the same tuple-counting formula also defines \(f_w\) for nongeodesic targets, whose counts are identically zero on every geodesic ray.
Since a label determines its destination, stationarity gives
\begin{align*}
	p_\mu(w)&=\sum_i\mu_iN_{i,s_1}\prod_{j=1}^{k-1}N_{s_j,s_{j+1}}\\
	&=\mu_{s_1}\prod_{j=1}^{k-1}N_{s_j,s_{j+1}}\\
	&=\frac{\ell_{s_1}r_{s_k}}{r_0\lambda^{k-1}}\prod_{j=1}^{k-1}M_{s_j,s_{j+1}}.
\end{align*}
The empty product for \(k=1\) is one, and indices such as \(N_{i,s}\) refer to the state \(v_s\).
The factor \(\lambda^{-(k-1)}\), rather than \(\lambda^{-k}\), appears because the initial sum already gives the stationary probability of the first destination.
For example,
\[
	p_\mu(aa)=\frac{\sqrt5}{10\lambda},\qquad p_\mu(bc)=\frac{5-\sqrt5}{20\lambda\varphi}=\frac{25-11\sqrt5}{20},\qquad p_\mu(cb)=0.
\]
Also, both transitions in \(ab^{-1}a^{-1}\) have probability \(1/\lambda\), while each continuation of a \(c\)-run has probability \(1/\lambda\), giving
\[
	p_\mu(ab^{-1}a^{-1})=\frac{\sqrt5}{10\lambda^2},\qquad p_\mu(c^4)=\frac{5-\sqrt5}{20\lambda^3}.
\]

For each fixed \(k\geq1\), consider the finite-state Markov chain of consecutive length-\(k\) path tuples in \(C\).
The restriction \(N_C\) is irreducible and aperiodic, so \(N_C^n\) has positive entries for all sufficiently large \(n\).
Choosing such a connecting length divisible by \(k\) joins the terminal vertex of any tuple to the initial vertex of any other tuple, proving that the tuple chain is irreducible.
Its stationary probabilities are the stationary path-cylinder probabilities of these tuples, and summing over the tuples spelling \(w\) gives \(p_\mu(w)\).
The finite-state Markov chain ergodic theorem used in Theorem~\ref{thm:arbitrary-length-ordered-frequencies} therefore gives convergence of the empirical tuple proportions, and hence of the ordered frequency of \(w\) to \(p_\mu(w)\), from any starting state in \(C\).
For the based law, discard the first length-\(k\) tuple: the remaining tuples start at time \(k\) in \(C\), retain their original alignment, and have the law of this chain from their entrance state.
Removing that single tuple does not change the limiting proportion.
Taking the countable intersection over finite words therefore gives
\[
	f_w(\sigma)=p_\mu(w)\qquad\text{for every nonempty word \(w\), for \(\mu^0\)-almost every \(\sigma\).}
\]
In particular, all nonempty words accepted by this geodesic combing have positive limiting frequency, but not all geodesic words in the group do: \(cb\) is geodesic and has frequency zero on every combing ray.

\paragraph{A finite empirical sample.}

A saved run of \texttt{geodesic\_combing\_zz2.ipynb} records one path of \(L=10\,000\,000\) letters, sampled from \(v_0\) using the transition probabilities from \(N\).
The targets \(a\), \(bc\), \(cb\), \(ab^{-1}a^{-1}\), and \(c^4\) are all measured on this same path.
For a target of length \(k\), it inspects the consecutive, nonoverlapping tuples at positions \(1,\ldots,k\), then \(k+1,\ldots,2k\), and so on, discarding a trailing incomplete tuple.
Thus the empirical denominator is \(\lfloor L/k\rfloor\), not \(L\), and the observed statistic is \(n_w(\sigma;\lfloor L/k\rfloor)/\lfloor L/k\rfloor\).
Table~\ref{tab:example-zz2-frequencies} compares selected saved observations with the exact stationary probabilities derived above.

\begin{table}[htbp]
	\centering
	\renewcommand{\arraystretch}{1.4}
	\begin{tabular}{c c c r r}
		\hline
		Word & Exact \(p_\mu(w)\) & Decimal & Matches/tuples & Empirical \\
		\hline
        & & &\\
		\(a\) & \(\dfrac{\sqrt5}{10}\) & \(0.223607\) & \(2\,237\,112/10\,000\,000\) & \(0.223711\) \\
		\(bc\) & \(\dfrac{5-\sqrt5}{20\lambda\varphi}\) & \(0.020163\) & \(100\,482/5\,000\,000\) & \(0.020096\) \\
		\(cb\) & \(0\) & \(0.000000\) & \(0/5\,000\,000\) & \(0.000000\) \\
		\(ab^{-1}a^{-1}\) & \(\dfrac{\sqrt5}{10\lambda^2}\) & \(0.012461\) & \(41\,638/3\,333\,333\) & \(0.012491\) \\
		\(c^4\) & \(\dfrac{5-\sqrt5}{20\lambda^3}\) & \(0.001818\) & \(4\,590/2\,500\,000\) & \(0.001836\) \\
        & & &\\
		\hline
	\end{tabular}
	\caption{Exact stationary probabilities and frequencies from the saved based sample of \(10\,000\,000\) letters.
	The decimal probabilities and empirical frequencies are rounded to six places; the displayed count ratios retain the observed integer data.}
	\label{tab:example-zz2-frequencies}
\end{table}

Figure~\ref{fig:example-zz2-convergence} displays the cumulative ordered frequencies on the same saved path, without sampling a new trajectory for any target.
For a word \(w\) of length \(k\), the curve records
\[
	\widehat f_w(m):=\frac{n_w(\sigma;m)}{m},\qquad 1\leq m\leq m_w:=\left\lfloor\frac{L}{k}\right\rfloor.
\]
The horizontal axis uses a logarithmic scale for the number \(m\) of complete, consecutive, nonoverlapping tuples, while the vertical axis is the cumulative frequency on a linear scale.
At the same horizontal coordinate, a target of length \(k\) uses the first \(mk\) letters, so different target lengths need not refer to the same path prefix.
Each curve ends at its own maximum tuple count \(m_w\), and its terminal dot is the corresponding empirical frequency in Table~\ref{tab:example-zz2-frequencies}.
The dashed horizontal lines show the exact stationary probabilities \(p_\mu(w)\), evaluated numerically for plotting.
Only the curve rendering is downsampled, to at most \(3\,000\) plotted points per target; every complete tuple is counted when computing the cumulative frequencies.

\begin{figure}[htbp]
	\centering
	\includegraphics[width=\linewidth]{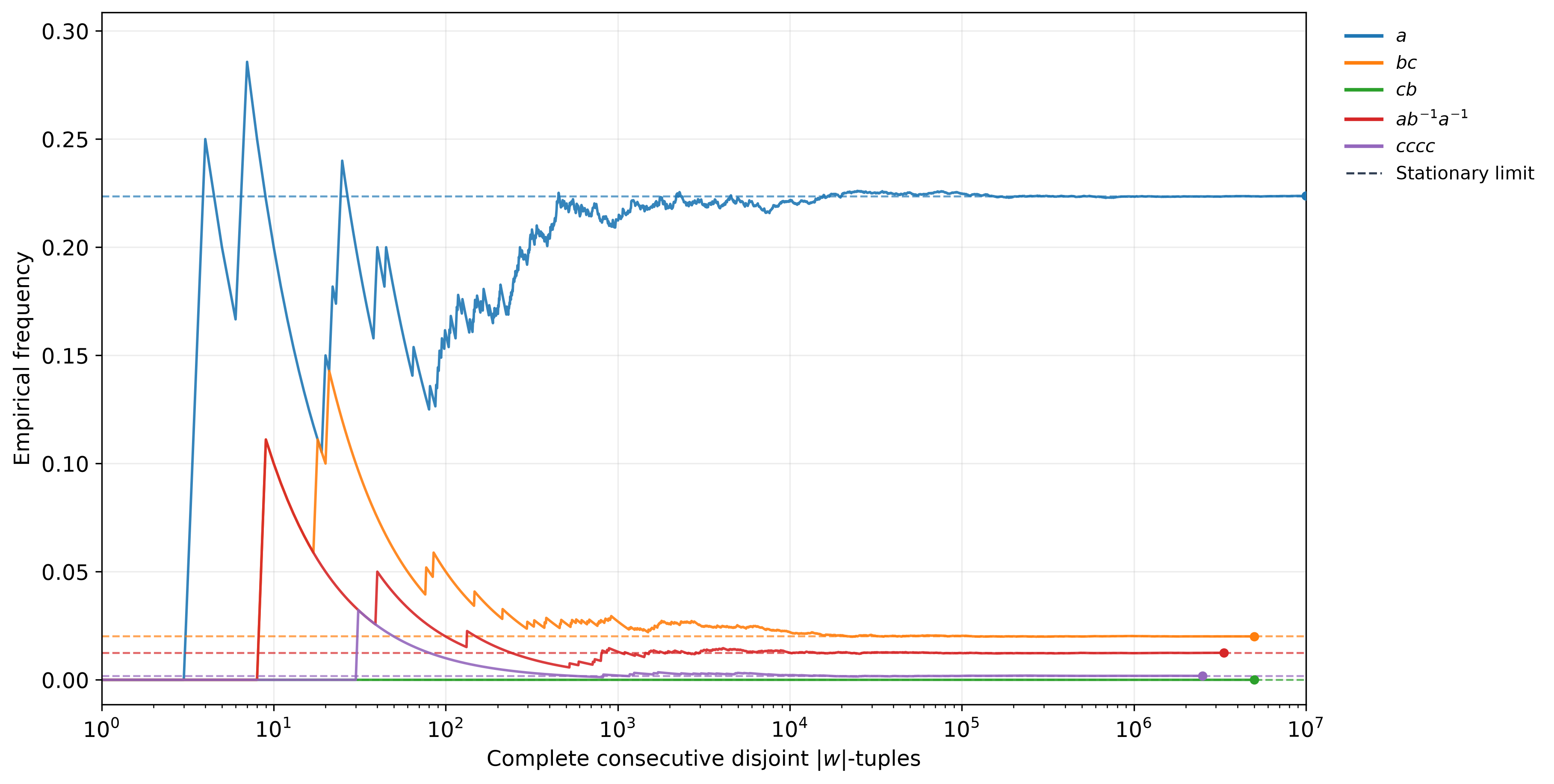}
	\caption{Cumulative ordered frequencies for the five targets on the same saved \(10\,000\,000\)-letter path.
	Solid curves show \(n_w(\sigma;m)/m\), dashed levels give \(p_\mu(w)\), and endpoint dots mark the final empirical values.
	The logarithmic horizontal axis counts complete tuples, so the curves terminate at different values \(m=\lfloor L/|w|\rfloor\).
	The drawing is downsampled, but the counts use all complete tuples.}
	\label{fig:example-zz2-convergence}
\end{figure}

These observations illustrate the limits on one finite sampled path.
The cumulative frequencies need not approach their limits monotonically, and successive tuples are not assumed independent.
The zero for \(cb\), by contrast, is forced by the missing transition rather than by sample size.

\newpage
\bibliographystyle{amsalpha}
\bibliography{references}

\end{document}